\documentclass[11pt,a4paper]{article}
\usepackage[margin = 0.9in,top=0.5in,headsep=9mm]{geometry}
\usepackage[utf8]{inputenc}
\usepackage[english]{babel}
\usepackage{xcolor}
\usepackage{amsmath}
\usepackage{amssymb}
\usepackage{nicefrac}
\usepackage{amsthm}
\usepackage{dsfont}
\usepackage{enumerate}
\usepackage{hyperref}
\usepackage[backend=bibtex,doi=false,url=false,firstinits=true]{biblatex}
\usepackage{tikz}
\usepackage{float}
\usepackage{esint}
\usepackage{mathtools}
\usepackage{csquotes}
\usepackage{siunitx}
 \usepackage{multirow}
\usepackage[singlelinecheck=false]{caption}
\numberwithin{equation}{section}

\newcommand{\R}{\mathbb{R}}

\newcommand{\N}{\mathbb{N}}

\newcommand{\eps}{\varepsilon}

\DeclareMathOperator*{\argmin}{arg\,min}
\newcommand{\abs}[1]{\left| #1 \right|}
\newcommand{\pars}[1]{\left( #1 \right)}

\newcommand{\seq}[1]{\pars{#1_n}_{n\in\N}}

\newcommand{\set}[1]{\left\{#1 \right\}}

\newcommand{\inv}[1]{#1^{-1}}
\newcommand{\norm}[1]{\left\Vert #1 \right\Vert }

\renewcommand{\div}{\text{ div}}

\newcommand{\inner}[2]{\left\langle #1, #2 \right\rangle}

\newcommand{\dom}{\text{dom }}

\newcommand{\prox}{\text{prox}}

\definecolor{unibluedark}{RGB}{68, 111, 128}
\definecolor{unibluelight}{RGB}{84, 159, 198}
\def\mathcolor#1#{\@mathcolor{#1}}
\def\@mathcolor#1#2#3{%
  \protect\leavevmode
  \begingroup
    \color#1{#2}#3%
  \endgroup
}
\newcommand{\tv}[1]{\abs{#1}_{\text{TV}}}

\newtheorem{theorem}{Theorem}[section]
\newtheorem{lemma}[theorem]{Lemma}
\newtheorem{corollary}[theorem]{Corollary}
\newtheorem{proposition}[theorem]{Proposition}
\newtheorem{definition}[theorem]{Definition}
\newtheorem{example}[theorem]{Example}
\newtheorem*{remark}{Remark}

\newcommand{\sign}{\mathrm{sign}}
\title{Multiscale hierarchical decomposition methods for ill-posed problems}
\author{Stefan Kindermann\footnote{ Industrial Mathematics Institute, Johannes Kepler University Linz, Austria, email: {\tt kindermann@indmath.uni-linz.ac.at}
}, Elena Resmerita\footnote{Institute of Mathematics, University of Klagenfurt, Austria, email: {\tt elena.resmerita@aau.at} } and Tobias Wolf\footnote{Institute of Mathematics, University of Klagenfurt, Austria, email: {\tt tobias.wolf@aau.at} }}
\newcommand{\xdag}{x^\dagger}
\begin{document}

\maketitle
\begin{abstract}
The Multiscale Hierarchical Decomposition Method (MHDM) was introduced in \cite{VeseMultiscale, VeseDeblurring} as an iterative method for total variation regularization, with the aim of recovering details at various scales from images corrupted by additive or multiplicative noise. Given  its success beyond image restoration, we extend the MHDM iterates in order to solve larger classes of linear ill-posed problems in Banach spaces. Thus, we define the MHDM for  more general convex or even non-convex penalties, and provide  convergence results for the data fidelity term. We also propose a flexible version of the method using adaptive convex functionals for regularization, and show an interesting multiscale decomposition of the data.
This decomposition result is highlighted for the  Bregman iteration method that can be expressed as an adaptive MHDM. Furthermore, we state necessary and sufficient conditions when the MHDM iteration agrees with the variational Tikhonov regularization, which is the case, for instance, for one-dimensional total variation denoising. Finally, we investigate several particular instances and perform numerical experiments that point out the robust behavior of the MHDM.
\end{abstract}
\section{Introduction}
In their influential works from 2004 and 2008, Tadmor, Nezzar and Vese \cite{VeseMultiscale, VeseDeblurring} introduced a multiscale decomposition method for image denoising, deblurring and segmentation, based on the popular total variation model of Rudin, Osher and Fatemi (ROF)   \cite{ROF}. Recall that ROF decomposes an image  $f\in L^2(\Omega)$ in cartoon and texture as $f=u_\lambda+v_\lambda$ such that
\begin{equation}\label{eq:rof}
(u_{\lambda},v_\lambda)=\arg\min_{u+v=f}\left\{\lambda\|v\|_{L^2}^2+\abs{u}_{TV}\right\},
\end{equation}
with $v_{\lambda}\in L^2(\Omega)$ and $u_{\lambda}\in BV(\Omega) =\set{u\in L^2(\Omega): \tv{u} < \infty}$, where $\tv{\cdot}$ denotes the total variation seminorm given by $\abs{u}_{TV}:= \sup\set{\int\limits_\Omega u \div \varphi : \varphi \in C_0^\infty(\Omega)\text{ and } \norm{\varphi}_\infty\le 1}$. Here
$\Omega$ is a bounded and open set in $\R^2$. %and $\abs{u}_{TV}:= \sup\set{\int\limits_\Omega u \div \varphi : \varphi \in C_0^\infty(\Omega)\text{ and } \norm{\varphi}_\infty\le 1}$ is the total variation seminorm. 
While the main features in natural images are very well restored via ROF, the texture at various scales might not be optimally recovered. The Multiscale Hierarchical Decomposition Method (MHDM) 
copes with this difficulty by decomposing an image into a  sum of multiple images, each of these containing features of the original image at a different scale. Thus, besides  extracting a cartoon representation  of the original image, it allows recovering  more oscillatory image components. One of the main advantages is the relatively simple modeling involved in the procedure. Instead of employing more complicated and numerically expensive penalty terms, the improvement is achieved by \enquote{zooming-in}: What is considered noise and texture at one scale, can be regarded as cartoon at a finer scale. The explicit decomposition of images into parts that contain increasingly more subtle features has made the method attractive  to solve plenty of other problems. Examples can be found in the fields of nonlinear partial differential equations \cite{tadmor2012hierarchical},  image registration \cite{Nachman, ImageRegistration},  graph theory \cite{MultiscaleGraphs}, compressed sensing, deconvolution of the Helmholtz filter and linear regression - see the PhD thesis \cite{zhong}.
 Note that the latter concerns actually a general MHDM, but for solving  linear inverse problems in finite dimension. Moreover, the approach  in \cite{Nachman}  corresponds to a wider range of possible applications apart from image registration since it employs nonlinear operators, not necessarily quadratic data fidelities, and powers of seminorms as penalties.

In the sequel, we formulate  the Multiscale Hierarchical Decomposition Method in Banach spaces and present the state of the art to the best of our knowledge.

 Let $X$ be a Banach space and $J:X \to [0,\infty]$ be a proper and lower semicontinuous functional which is bounded from below. Let $T \in \mathcal{L}(X,H)$ be an ill-posed linear operator with values in a Hilbert space $H$, and  fix $f\in H$. Generally, we are interested in solving \begin{equation}\label{eq:J_minimizing_sol}
    \begin{cases}
    \text{minimize  } &J(x)\\
    s.t.  &Tx = f .
    \end{cases}
\end{equation}
We will furthermore assume that problem \eqref{eq:J_minimizing_sol} is non-degenerate, that is, there exists an element $\xdag \in \dom J:= \set{x \in X: J(x)<\infty}$ such that \begin{equation}\label{eq:xdagger}
    T\xdag = f.
\end{equation}
Throughout this work, we assume that the generalized Tikhonov functional
\begin{equation}\label{eq:Tikhonov_functional}
    F_\lambda(u) := \frac{\lambda}{2}\norm{Tu-y}^2 + J(u) 
\end{equation} admits a minimizer for all $\lambda>0$ and $y \in H$. Standard conditions when this holds true can be found, e.g., in \cite{HoKaPoSc07}.
Moreover, we denote a minimizer of 
the generalized Tikhonov functional 
as $x_\lambda$, i.e., 
\begin{equation}\label{eq:tikhonov}
x_\lambda := \argmin_{u \in X} F_\lambda (u) = 
 \argmin_{u \in X} \frac{\lambda}{2}\norm{Tu-y}^2 + J(u).
\end{equation}

We approach \eqref{eq:J_minimizing_sol} by the Multiscale Hierarchical Decomposition Method (MHDM) that  works as follows: Choose a sequence $(\lambda_n)_{n \in \N_0}$ of positive real numbers and compute \begin{equation}\label{eq:MHDM_initial_step}
    u_0 \in \argmin\limits_{u \in X}\frac{\lambda_0}{2}\norm{Tu-f}^2 + J(u).
\end{equation}
Denote the residual $v_0 = f -Tu_0$ and set $x_0 = u_0$. Next, compute iteratively for $n = 1,\dots$ 
\begin{equation}\label{eq:MHDM_step}
\begin{split}
    u_n &\in \argmin\limits_{u \in X} \frac{\lambda_n}{2}\norm{v_{n-1}-Tu}^2 + J(u),\\
    x_n &= x_{n-1} +u_n\qquad\mbox{and}\\
    v_n &= v_{n-1} -Tu_n = f-Tx_n.
    \end{split}
\end{equation}
The resulting sequence $(x_n)_{n \in \N_0}$ with $x_n = \sum_{i=1}^n u_i$ is considered as an  approximation of $x^\dagger$, thus yielding a scale decomposition
depending on the choice of 
$J$ and $\lambda_n$.

Note that  \eqref{eq:MHDM_step} can be rewritten as
\begin{eqnarray}\label{eq:MHDM_step_1}
    x_n \in \argmin\limits_{x \in X} \frac{\lambda_n}{2}\norm{Tx-f}^2 + J(x-x_{n-1}).
\end{eqnarray}
If  $X$ is a Hilbert space and $J=\frac{\|\cdot\|^2}{2}$, this procedure reads
\begin{eqnarray}\label{eq:nonstationary_tikhonov}
    x_n \in \argmin\limits_{x \in X} \lambda_n\norm{Tx-f}^2 + \|x-x_{n-1}\|^2,
\end{eqnarray}
which is the nonstationary Tikhonov regularization---see \cite{groetsch, HankeGroesch} for detailed convergence and error estimate results, as well as \cite{ScherzerInvScale} concerning the inverse scale space method as an asymptotic formulation of the method. In fact, one recognizes in \eqref{eq:nonstationary_tikhonov} the particular quadratic setting for another prominent approach of inverse problems,  namely  the nonstationary augmented 
Lagrangian  method \cite{frick_scherzer} known also as  the Bregman iteration \cite{orig_breg}. 
Note, however, that these methods usually differ from the MHDM method for non-quadratic 
penalties $J$.
Before describing the latter, let us recall some definitions from convex analysis. If the functional $J$ is convex, one defines the subgradient of $J$ at a point $x_0\in \dom J$ as \begin{equation}\label{eq:subgradient}
    \partial J(x_0) = \set{x^* \in X^* : \inner{x^*}{x-x_0}\le J(x) -J(x_0) \text{ for all } x \in X},
\end{equation}
where $X^*$ stands for the dual space of $X$. Furthermore, for any point $x_1 \in \dom J$, the Bregman distance of $x_0$ and $x_1$ with respect to $x^* \in \partial J(x_0)$ is denoted by \begin{equation}\label{eq:Bregman_Distance}
    D_J^{x^*}(x_1,x_0) = J(x_1) -J(x_0) -\inner{x^*}{x_1-x_0}.
\end{equation}
Now we are in a position to recall the Bregman iteration:  For some sequence $(\lambda_n)\subset (0,\infty)$ and for any $n\in\N$, let
\begin{equation}\label{eq:bregman} 
 x_n \in \argmin_{x\in X} \frac{\lambda_n}{2}  \|T x -f\|^2 +
D_J^{p_{n-1}}(x,x_{n-1}), 
\end{equation}
with $x_0=0$ and $p_0=0$, where in each step one chooses
\[ p_{n} = \lambda_nT^*(f-Tx_n) +p_{n-1} \in \partial J(x_{n}).\] 
Here $T^*: H \to X^*$ stands for the adjoint of the operator $T$.
While for  the related method \eqref{eq:bregman} comprehensive convergence results exist in the corresponding references mentioned above, the situation is different for the MHDM defined with a non-quadratic penalty $J$ in infinite-dimensional spaces. Intriguingly,  there is no convergence result  for the sequence of iterates $(x_n)$ apart from  the denoising case  (when $T$ is the identity), which is a consequence of the weak/strong convergence  for  the residual $(Tx_n-f)$ (cf. \cite{VeseMultiscale, VeseDeblurring, Nachman}) or of the residual error estimates in \cite{MultiscaleRefinementImaging}. As regards the penalty $J$,  this was assumed so far to be a (power of a) seminorm. 

Another interesting open question was raised in \cite{VeseMultiscale}, whether the  MHDM iterate $x_n$  coincides with the solution $x_{\lambda_n}$ of Tikhonov regularization \eqref{eq:tikhonov} corresponding to the parameter $\lambda_n$. %

In general, this is not the case. However, there are situations when the answer is positive, as can be seen in section \ref{sect:ComparisonTikhonov}.

Thus, the main contributions of this study are as follows. First of all, we extend the existing convergence results regarding  the residual to the case when the penalty is a more general convex function than a (power of a) seminorm or when it belongs to a class of nonconvex functions. Furthermore,  we propose a generalization of MHDM by empowering the penalty to be  adaptive, and point out a couple of specific penalties that yield well known methods for solving \eqref{eq:J_minimizing_sol}. For instance, as a  side result that is interesting in itself, we formulate the Bregman iteration \eqref{eq:bregman} as a generalized MHDM with appropriate adaptive penalties, and obtain a curious multiscale norm decomposition of the data in terms of $(Tu_n)$ and symmetric Bregman distances $D_J^\text{sym}(x_n,x_{n-1})$. Since the topic of generalized MHDM with   new, meaningful, adaptive penalties is quite challenging, we will consider  it for future research. Moreover, we state necessary and sufficient conditions for the MHDM to agree with the Tikhonov regularization. We verify these conditions for one-dimensional TV-denoising, as well as  for particular situations in two-dimensional TV-denoising and for  finite-dimensional $\ell^1$-regularization. In particular, we emphasize that the  so-called positive cone condition \cite{Efron04} for the operator $T$ in
the $\ell^1$-regularization case does allow one to compare  the  MHDM iterate $x_n$ to  the solution  $x_{\lambda_n}$ of  Tikhonov regularization \eqref{eq:tikhonov} corresponding to the  regularization parameter $\lambda_n$ used in the MHDM.
Last but not least, we focus on several examples to understand how MHDM performs theoretically and computationally. The numerical experiments show a robust behavior of the MHDM with respect to  the involved parameters. Moreover, they give us hope that, under suitable assumptions which need to be found, the Multiscale Hierarchical Decomposition Method does converge in a more general framework, that is when $T\neq Id$, where $Id$ denotes the identity operator.

The structure of this work is as follows. Section \ref{sect:Convergence} presents the convergence of the residual in the case of some convex or even nonconvex penalties. Thereafter, we suggest in Section \ref{sect:extensions} a generalization of the iterative method and derive a decomposition result of the data. The comparison of the MHDM and the generalized Tikhonov regularization can be found in Section \ref{sect:ComparisonTikhonov}, while  Section \ref{sect:Numerical} presents the numerical experiments.

\section{Convergence of the residual}\label{sect:Convergence}

As mentioned in the introduction, no general result regarding the convergence of the MHDM iterates has been shown  when the problem $Tx = f$ is ill-posed. However, if the problem is well-posed, that is, $T$ is a bijective linear operator with continuous inverse, convergence of the iterates  $(x_n)$ is a consequence of the convergence of the residual. For the setting when $J$ is a power of a seminorm, it was shown in  \cite{Nachman} that the (not necessarily quadratic)  residual converges. Nonetheless, by adding the  term $J(x)$ to \eqref{eq:MHDM_step_1}, in \cite{Nachman} it is proved that the iterates 
$(x_n)$ of the resulting  {\em tight MHDM} 
(cf.~\eqref{eq:tight_MHDM} below) converge on subsequences to  a solution of \eqref{eq:J_minimizing_sol}. Similar  results  were derived in \cite{MultiscaleRefinementImaging} for the iterates of a    refined version of the tight MHDM that uses two different penalties, namely  $J$ for the components $u_n$ and $R$ for the sum $x_n$ of the components $u_n$. Note that \cite{MultiscaleRefinementImaging} established also error estimates for the residual of  the MHDM and of its tighter versions.

In this section, we extend the latter result  under the assumption of a generalized triangle inequality on the penalty $J$. This will lead to convergence rates of the residual for  large classes of penalty functions, including certain nonconvex functions.  Additionally, we establish a weak convergence result for the residual in a complementary case of general convex functions (which do not necessarily satisfy  a generalized triangle inequality).

 \subsection{The case of exact data}
We start by showing the results under the assumption of exact data $f$.
\begin{theorem}\label{thm:Convergence_results} Let $(x_n)_{n \in \N_0}$ be the sequence generated by \eqref{eq:MHDM_initial_step}--\eqref{eq:MHDM_step}, and let
 \eqref{eq:xdagger} hold.
\begin{enumerate}[(i)]
\item  Assume that $J$ is minimal at $0$ and that there is $C \ge 1$ such that  \begin{equation}\label{eq:Generalized_triangle_inequality}
        J(x-y) \le C (J(x) +J(y))
    \end{equation} for all $x,y \in X$. 
If $\lambda_n$ is chosen such that $2C \lambda_{n-1} \le \lambda_n$ for all $n\in \N$, then the residual satisfies \begin{equation}\label{eq:Convergence_rate_normpowers}
    \norm{f -Tx_n} \le \pars{4C\frac{J(x^\dagger)}{\lambda_0(n+1)}}^\frac{1}{2}
\end{equation}
for all $n \in \N_0$.
\item  Assume $J$ is convex and $\dom J$ is dense in $X$. Moreover, let one of the following conditions hold. \begin{enumerate}[a)]
\item If $J$ is minimal at $0$, then the residuals are 
monotonically decreasing.
\item If $J(0)<\infty$ and $\lambda_n$ is chosen such that \begin{equation}\label{eq:Series_lambdas_finite}
     \sum_{n = 0}^\infty \frac{1}{\lambda_n} <\infty,
 \end{equation}
 then the residuals 
 satisfy
 \[ \|f -T x_n\|^2  \leq  2 \tilde C \sum_{j=0}^n \frac{1}{\lambda_n}, \]
 where $\tilde C = J(0) -\inf\limits_{x\in X} J(x)$. 
\end{enumerate}
In both cases  (a) and (b),  $((f-Tx_n))_{n \in \N_0}$ is bounded, and every weak limit point is in the kernel of $T^*$. In particular, this means that $T^*(f-Tx_n)$ converges to $0$ in the weak-*-topology of $X^*$.
\end{enumerate}
\begin{proof}
\begin{enumerate}[(i)]

%-------------------- nonconvex case -------------------
\item Let $n \in \N$. By the optimality of $u_n$ defined in \eqref{eq:MHDM_initial_step} and $\eqref{eq:MHDM_step}$  by comparing to $u = 0$ it holds
\begin{equation}\label{eq:Compare_to_0}
    \frac{\lambda_n}{2} \norm{f-Tx_n}^2 + J(u_n) \le \frac{\lambda_n}{2}\norm{f-Tx_{n-1}}^2+J(0).
\end{equation}
On the other hand, comparing to $u = x^\dagger-x_{n-1}$, we obtain \begin{equation}\label{eq:Compare_to_error}
    \frac{\lambda_n}{2} \norm{f-Tx_n}^2 + J(u_n) \le J(x^\dagger -x_{n-1}).
\end{equation}
In particular, the minimality of $J$ at $0$ and \eqref{eq:Compare_to_0} imply that $\norm{f-Tx_n}$ is decreasing. Using \eqref{eq:Generalized_triangle_inequality},  it holds for $k \in \N$, 
\begin{align*}
    \frac{\lambda_k}{2}\norm{f-Tx_k}^2 + J(x^\dagger-x_k)  &= \frac{\lambda_k}{2}\norm{f-Tx_k}^2 + J(x^\dagger-x_{k-1} -u_k) \\&\le \frac{\lambda_k}{2}\norm{f-Tx_k}^2 + C\pars{J(x^\dagger-x_{k-1}) + J(u_k)}\\&=\frac{\lambda_k}{2}\norm{f-Tx_k}^2 + J(u_k) + (C-1) J(u_k) + CJ(x^\dagger-x_{k-1}) \\&\overset{\eqref{eq:Compare_to_error}}{\le} (C-1) J(u_k) + (C+1) J(x^\dagger-x_{k-1})\\&\le 2CJ(x^\dagger-x_{k-1}),
    \end{align*}
where the last inequality follows from \eqref{eq:Compare_to_error} and $\frac{\lambda_k}{2}\norm{f-Tx_k}^2\ge 0$. By using this and the choice of $\lambda_k$, we can conclude\begin{equation}\label{eq:Recursion_estimate_powers_of_norms}
    \frac{1}{2}\norm{f-Tx_k}^2 + \frac{J(x^\dagger-x_k)}{\lambda_k} \le \frac{2C}{\lambda_k}J(x^\dagger-x_{k-1}) \le \frac{1}{\lambda_{k-1}}J(x^\dagger-x_{k-1}).
\end{equation} Thus, we can repeatedly use \eqref{eq:Recursion_estimate_powers_of_norms} for $k=0,\dots,n$, sum up and obtain\begin{equation*}
     (n+1)\frac{\norm{f-Tx_n}^2}{2} + \frac{J(x^\dagger-x_n)}{\lambda_n} \le \sum\limits_{k = 0}^n\left(\frac{\norm{f-Tx_k}^2}{2}\right) + \frac{J(x^\dagger-x_n)}{\lambda_n}\le \frac{2C}{\lambda_0}J(x^\dagger),
\end{equation*}
where the first inequality is implied by the monotonicity of the residual.

This yields \eqref{eq:Convergence_rate_normpowers}.

%----------------convex case -----------------------
\item Let $n \in \N$. Then by  the optimality of $u_n$, we obtain \begin{equation}\label{eq:Compare_to_0_convex}
    \frac{\lambda_n}{2} \norm{f-Tx_n}^2 + J(u_n) \le \frac{\lambda_n}{2}\norm{f-Tx_{n-1}}^2 + J(0).
\end{equation}
This means $\norm{f-Tx_n}^2 \le \norm{f-Tx_{n-1}}^2 + 2\frac{J(0)-J(u_n)}{\lambda_n}$. If now $J(0)$ is minimal, this means that $\norm{f-Tx_n}$ is decreasing. Otherwise, we inductively arrive at \begin{equation*}
    \norm{f-Tx_n}^2 \le \norm{f}^2 + 2\sum_{k = 0}^n\frac{J(0)-J(u_k)}{\lambda_k}\le \norm{f}^2 + 2\sum_{k = 0}^n\frac{J(0)}{\lambda_k}.
\end{equation*}
By \eqref{eq:Series_lambdas_finite}, this implies that in both cases the sequence $(f-Tx_n)$ is bounded. Therefore, it admits a weakly convergent subsequence $(f-Tx_{n_k})$. Let $v$ be its limit. Note that by the optimality condition of $u_n$ and due to the convexity of $J$, it holds \begin{equation*}
\lambda_nT^*(f-Tx_n) \in \partial J(u_n).
\end{equation*}
This means \begin{align}\label{eq:Upper_bound_of_inner_product}
    \inner{T^*(f-Tx_n)}{z}_X&\le\frac{J(z)-J(u_n)}{\lambda_n} +\inner{f-Tx_n}{Tu_n}_H \nonumber\\&\le \frac{J(z)}{\lambda_n} + \norm{f-Tx_n}\norm{T(x_n-x_{n-1})} \nonumber\\&\le \frac{J(z)}{\lambda_n}+K,
\end{align}
 for some $K \ge 0$, because $\norm{f-Tx_n}$ and therefore $\norm{Tx_n}$ are bounded. Considering now \eqref{eq:Upper_bound_of_inner_product} with $x_n$ replaced by the subsequence $x_{n_k}$ and letting $k \to \infty$ yield \begin{equation}\label{eq:Contradiction_inner}
    \inner{T^*v}{z}\le K
\end{equation}
for all $z$ with $J(z)<\infty$. Now assume $T^*v \ne 0$. Then, for all $n\in \N$, there is $ z_n$ with $\inner{T^*v}{ z_n}\ge n$. Fix $\eps >0$. By the density of $\dom J$ which is equivalent to the density in the weak topology by Mazur's Lemma \cite[p.~6]{Ekeland}  (as  $\dom J$ is convex), we can find $z_n^\eps$ with $J(z_n^\eps) < \infty$ and $\abs{\inner{T^*v}{ z_n- z_n^\eps}} \le \eps$. Thus $\inner{T^*v}{z_n^\eps} \ge n-\eps$ holds, contradicting  \eqref{eq:Contradiction_inner}. Since this reasoning can be applied to any subsequence, the claim follows.

\end{enumerate}
\end{proof}
\end{theorem}

\begin{remark} \ 

\begin{enumerate}
    \item  If $J(x) = \abs{x}_{se}^p$ for a seminorm $\abs{\cdot}_{se}$ and $p\in (0,\infty)$, then condition \eqref{eq:Generalized_triangle_inequality} holds for $C = 2^{\max{\set{0,\,p-1}}}$. Furthermore, if $\seq{p}$ is a bounded sequence in $(0,\infty)$, then $J(x) = \sum\limits_{n = 1}^\infty \abs{x_n}^{p_n}$ satisfies \eqref{eq:Generalized_triangle_inequality} with $C = 2^{\max{\set{0,\,\sup p_n-1}}}$. The reader is referred to \cite{lor_res} for promoting sparsity by employing the latter functional $J$.
    \item In the situation of part $(i)$ in the previous theorem, strong convergence of the residual is also obtained if $(\lambda_n)_{n\in \N_0}$ is an arbitrary sequence increasing to $\infty$ (a proof is given for a more general version of the algorithm in Lemma \ref{lemma:Convergence_Residual_Generalization}). However, we do not obtain convergence rates without additionally assuming the rate of increase to be at least geometric.
    \item Estimate \eqref{eq:Convergence_rate_normpowers} also holds for general distance functions: Let $d:H\times H \to [0,\infty)$ be a function such that $d(x,x) = 0$ for all $x \in X$. If we replace the Hilbert space norm in \ref{eq:MHDM_initial_step} and \eqref{eq:MHDM_step} by $d$, i.e., we  consider the iteration \begin{equation*}
        u_n \in \argmin\limits_{u \in X} \lambda_n d(Tu, v_{n-1}) + J(u),
    \end{equation*}
    then, under the same assumptions as in part $(i)$ of \ref{thm:Convergence_results}, the estimate \begin{equation}\label{eq:COnbvergence_residual_general_penalty}
        d(Tx_n,f) -\inf\limits_{x \in X}d(Tx,f) \le 2C\frac{J(x^\dagger)}{\lambda_0(n+1)}
    \end{equation} holds. This means, we have found convergence rates of the residual for all penalties considered in 
    Section~2 of \cite{Nachman}. Furthermore, by making minor adaptions in the notation of the proof, \eqref{eq:COnbvergence_residual_general_penalty} also holds for nonlinear operators $T$. However, for nonconvex data fidelity terms or nonlinear operators, computing global minimizers of the corresponding Tikhonov functionals might be impossible or too expensive for practical applications. In those cases, one might consider alternative ways to regularize the ill-posed problem.
\end{enumerate}
\end{remark}

\subsection{The noisy data case}
Assume now that instead of exact data $f$, we are given a noisy measurement $f^\delta$. Typically, regularization methods behave semi-convergent in this situation. This means that the true solution is approached initially, but then the distance between the iterates and the true solution eventually increases. To overcome this issue, the iteration is terminated early according to some meaningful rule. We will use the discrepancy principle as a stopping criterion and show convergence of the residual as the noise level $\delta$ approaches $0$.

Let us assume the case of additive Gaussian noise, i.e. \begin{equation}\label{eq:noise}
    \norm{f^\delta -f} \le \delta
\end{equation} with some $\delta >0$. For $n \in \N$, denote by $x_n^\delta = \sum\limits_{j = 0}^n u_j^\delta$ the iterate of the MHDM with data $f^\delta$ instead of $f$. 
\begin{lemma}\label{lemma:Convergence_for_noise}
Assume that $f$ and $f^\delta$ verify \eqref{eq:noise}. If $J$ satisfies \eqref{eq:Generalized_triangle_inequality} and the sequence $(\lambda_n)_{n \in \N_0}$ is chosen such that $2C\lambda_{n-1}\le \lambda_n$, then the following holds, for all $n \in \N$: 
\begin{equation}\label{eq:Residual_rate_noisy}
    \norm{Tx_n^\delta-f^\delta}^2 \le 4C\frac{J(x^\dagger)}{\lambda_0(n+1)} + \delta^2.
\end{equation}

\begin{proof}
Let $n \in \N$. By the optimality of $u_n^\delta$ for \eqref{eq:MHDM_step} with $f$ replaced by $f^\delta$, it is \begin{align*}
    \frac{\lambda_n}{2}\norm{f^\delta -Tx_n^\delta}^2 + J(u_n^\delta) \le \frac{\lambda_n}{2}\norm{f^\delta - f}^2 + J(x^\dagger - x_{n-1}^\delta) \le \frac{\lambda_n}{2}\delta^2 + J(x^\dagger-x_{n-1}^\delta).
\end{align*}
By using analogous reasoning as in the proof of Theorem \ref{thm:Convergence_results} (i), the claim follows.
\end{proof}
\end{lemma}
We consider the following discrepancy principle: Choose some $\tau >1$ and let the index \begin{equation}\label{eq:Discrepancy_principle}
    n^*(\delta) = \max\set{n \in \N: \norm{Tx_n^\delta-f^\delta}^2 \ge \tau \delta^2}.
\end{equation}
Note that $n^*(\delta)$ is well-defined by \eqref{eq:Residual_rate_noisy}.

\begin{theorem}
If under the assumption of Lemma \ref{lemma:Convergence_for_noise}, the iteration is stopped at index $n^*(\delta) +1$, then $Tx_{n^*(\delta)+1}^\delta\to f$ as $\delta \to 0$.
\begin{proof}
By \eqref{eq:Discrepancy_principle}, it is \begin{equation*}
    \norm{Tx_{n^*(\delta)+1}^\delta -f^\delta} \le \tau^\frac{1}{2}\delta.
\end{equation*}
Therefore, the estimate \begin{equation*}
    \norm{Tx_{n^*(\delta)+1}^\delta -f} \le \norm{Tx_{n^*(\delta)+1}^\delta -f^\delta} + \norm{f-f^\delta} \le \tau^\frac{1}{2}\delta + \delta
\end{equation*}
yields the result.
\end{proof}
\end{theorem}

\begin{remark}
In the situation of part $(ii)$ in Theorem \ref{thm:Convergence_results}, we also get that $T^*(f^\delta-Tx_n^\delta)$ converges to $0$ in the weak-*-topology. Yet, it is not clear how to define a meaningful stopping index, as there are no convergence rates available.
\end{remark}

\subsection{Maximum entropy regularization}\label{Sec:Max_entropy}
In this section, we give an example where the conditions of Theorem~\ref{thm:Convergence_results} are not satisfied 
and the residual of the MHDM does not converge. 

Let $\Omega \subset \R^n$ be bounded. The negative Boltzman-Shannon entropy is defined as $S:L^1(\Omega) \to \R\cup\set{\infty}$, \begin{equation}\label{eq:entropy}
        S(x) = \begin{cases}
    \int_\Omega x(t) \log x(t) dt &\text{ if $x \ge 0$ a.e. and $x\log x \in L^1(\Omega)$}\\
    \infty &\text{ else.}
    \end{cases}
\end{equation}
Here we use the convention $0\log 0 = 0$. This functional can be employed in regularization methods to enforce non-negativity of approximate solutions - see, e.g., \cite{MaxEntropyRegularization, EnglLandlMaxEntropy, resmerita_anderssen2007}, as well as the survey \cite{cla-kal-res}. For an operator $T \in \mathcal{L}(X,H)$, the variational regularization \eqref{eq:tikhonov} with penalty $J=S$ is well-defined due to the coercivity and lower semicontinuity of the entropy with respect to the weak topology of $L^1(\Omega)$. This functional is of particular interest in our context, because it does not satisfy the conditions of Theorem~\ref{thm:Convergence_results}.

\begin{lemma}
Let $X = L^1(\Omega)$ and let $J = S$. Then $J$ does not satisfy \eqref{eq:Generalized_triangle_inequality}, for any $C \geq 1$.

\begin{proof}
Let $x(t) = y(t) = e^{-1}$ for all $t \in \Omega$. Since $x,y\in \dom S$ and $S(x-y) = 0$, one has
\begin{equation*}
    S(x) + S(y) = -2\int\limits_{\Omega}e^{-1}dt <0.
\end{equation*}
 Thus, \eqref{eq:Generalized_triangle_inequality} can not hold for any $C\ge 1$.

\end{proof}
\end{lemma}
Since the domain of $S$ is strictly contained in $L^1_+(\Omega) := \set{u \in L^1 : u\ge 0 \text{ a.e.}}$, the density assumption in Theorem~\ref{thm:Convergence_results} does not hold either. Thus, the theorem can not be applied to the MHDM  defined with the entropy penalty. In fact, one can show that the residual obtained in this setting  does not necessarily converge. In order to see this, e.g., for the simple example of the identity operator, we recall the proximal mapping of the negative entropy. Using Table~2 in \cite{CombettesSplitting} and Proposition~12.22 in \cite{BauschkeCombettesConvex} pointwise, one has 
\begin{equation}\label{eq:Proximal_entropy}
    prox_{\frac{S}{\lambda}}(y) = \argmin\limits_{x\in X} \frac{\lambda}{2}\norm{x-y}^2 + S(x) =  \frac{1}{\lambda}W\pars{\lambda\exp\pars{\lambda y -1}}
\end{equation}
where $W$ denotes the principal branch of the Lambert $W$ function.

\begin{lemma}\label{lemma:MHDM_entropy_nonconvergence}
Let $X = H = L^2(\Omega)$ and $T = Id$. Let furthermore $(\lambda_n)_{n \in \N_0}$ be a sequence of increasing positive numbers, and assume that $x^\dagger$ satisfies $0\le x^\dagger< \frac{1}{e}$ on a set $E$ with positive measure. Then the MHDM iterates  with $S$ (restricted to $L^2$)  as penalty term  will be bounded away from $x^\dagger$ in the following sense:\begin{equation*}
    x_n(t) > \dots >x_0(t) > x^\dagger(t), \quad \text{for all } t \in E \text{ and all } n \in \N.
\end{equation*}
\begin{proof}
Since $T  =Id$, it is $u_n = \prox_{\frac{S}{\lambda_n}}(x^\dagger-x_{n-1})$ for $n \in \N_0$ and $x_{-1} = 0$. We can therefore use a one-dimensional calculation.  Let $0 \le x < \frac{1}{e}$ and $\lambda\ge 0$. Multiplying the inequality \begin{equation*}
    \lambda \inv{e} > \lambda x
\end{equation*} with $\exp\pars{\lambda x}$ and applying $W$ (which is strictly increasing and positive on $[0, \infty)$)  yield \begin{equation*}
    W\pars{\lambda \exp(\lambda x - 1}) > W(\lambda x \exp(\lambda x)) = \lambda x,
\end{equation*} or equivalently \begin{equation*}
    \frac{1}{\lambda}W\pars{\lambda\exp\pars{\lambda x -1}} > x.
\end{equation*}
 Thus, because $x^\dagger(t) < \inv{e}$ on $E$,  the previous inequality with $x^\dagger (t)$ instead of $x$ and equation \eqref{eq:Proximal_entropy} 
 imply 
 \begin{equation*}
     x_0(t) =  u_0(t) = \pars{\prox_{\frac{S}{\lambda_0}}(x^\dagger)}(t)  > x^\dagger(t),
\end{equation*}
independently of the choice of $\lambda_0$, for all $t\in E$. Since by \eqref{eq:Proximal_entropy} the increment \begin{equation*}
    x_1-x_0 = u_1 = \prox_{\frac{S}{\lambda_1}}(x^\dagger-x_0)
\end{equation*} is positive, we must have $x_1(t) > x_0(t) > x^\dagger(t)$ for all $t \in E$. The statement now follows by induction.
\end{proof}
\end{lemma}

The previous Lemma shows that we cannot expect $L^2$-convergence of the residual (which in this case is the same as convergence of $(x_n)$) for the MHDM with entropy penalty, if $x^\dagger<\inv{e}$ on a set of positive measure. In fact, the residual does not converge even if the ground truth is bounded away from $\frac{1}{e}$. Indeed, assume $(x_n)$ does converge to $x^\dagger$. If some iterate $x_n$ satisfies $x_n>x^\dagger$ on a set of positive measure, then \eqref{eq:Proximal_entropy} implies that $x_{n+1} > x_n$ on that set, meaning that $x_n$ cannot converge to $x^\dagger$. Otherwise, if $x_n \le x^\dagger$ a.e. for all $n\in \N_0$, then convergence would yield $0\le x^\dagger-x_{n_0}\le \inv e$  on a set $E$ on positive measure for some $n_0\in \N_0$. Using the same reasoning as in the proof of Lemma \ref{lemma:MHDM_entropy_nonconvergence} but with $x^\dagger-x_{n_0}$ instead of $x^\dagger$, one would obtain $u_{n_0+1} > x^\dagger - x_{n_0}$ on $E$. But this would be equivalent to $x_{n_0+1}> x^\dagger$ on $E$, which contradicts $x_{n_0+1} \le x^\dagger$ a.e.
This negative convergence result for the residual can be explained by the fact that the domain of $S$ is not dense in $L^2(\Omega)$, which consequently does not allow the MHDM to iteratively adjust the approximation with each step. Therefore, any kind of convergence we can hope to achieve will be different from pointwise or (weak) $L^2$ convergence.

\section{Extension of the algorithm with flexible penalty terms}\label{sect:extensions}

The idea of the MHDM as defined in \eqref{eq:MHDM_initial_step} and \eqref{eq:MHDM_step} is to look for solutions 
of \eqref{eq:xdagger} 

that show similar behavior on different scales.  In this section, we present a more flexible version of the MHDM, which aims to recover solutions with different behavior on different scales. To this end, we introduce a scheme with more general penalty terms. Consider a sequence $(J_n)_{n \in \N_0}$ of functionals on $X$ and define a sequence of approximate solutions by computing

\begin{equation}\label{step flexible}
    u_n \in \argmin\limits_{u \in X} \frac{1}{2}\norm{v_{n-1} -Tu}^2+ J_n(u),
\end{equation}
with $x_n = \sum\limits_{k = 0}^n u_k$, $v_n = f-Tx_n$
and $x_{-1} =0$ as before. 
In particular, the choice $J_n = \frac{1}{\lambda_n}J$ for some fixed $J$ yields the original MHDM. 

In the sequel, we show an interesting norm decomposition  of the data, as well as convergence of the residual for this generalized MHDM. Then, we point out a couple of special choices for the functionals $J_n$ which yield  known iterative methods for solving \eqref{eq:J_minimizing_sol}.

\subsection{Multiscale norm decomposition of the data}
 Let us start with a decomposition result for the norm of the data $f$ which, adapted to TV-deblurring, can be found in Theorem 2.8 of \cite{VeseDeblurring}. Due to the flexibility of the penalty terms considered for this extension, the result can be transferred to related iterative schemes, as we will illustrate below. To simplify notation, we define  
\begin{equation}\label{eq:zeta}
\zeta_k := T^*(f-Tx_k) = T^*v_k \in \partial J_k(u_k)
\end{equation}
for $k \in \N_0$. 
\begin{theorem}\label{thm:Decompostion_finite_equality}
Let $(J_n)_{n \in \N_0}$ be a sequence of proper, convex, lower-semicontinuous functions such that a sequence $(u_n)_{n \in \N_0}$ is well-defined via \eqref{step flexible}. Then for any $n \in \N_0$, one has \begin{equation}\label{eq:decomposition}
    \norm{f}^2 =\norm{v_n}^2+ \sum\limits_{k = 0}^n\pars{\norm{Tu_k}^2 + 2\inner{\zeta_k}{u_k}}.
\end{equation}
\begin{proof}
For $k \in \N_0$, it is $v_{k-1} = v_{k} +Tu_k$. Therefore,
\begin{equation*}
\norm{v_{k-1}}^2 = \norm{v_{k}}^2 + \norm{Tu_k}^2 + 2\inner{v_k}{Tu_k} = \norm{v_{k}}^2 + \norm{Tu_k}^2 + 2\inner{\zeta_k}{u_k}.
\end{equation*}
Telescoping, we obtain \begin{equation*}
    \sum\limits_{k = 0}^n \pars{\norm{Tu_k}^2 + 2\inner{\zeta_k}{u_k}} = \sum\limits_{k = 0}^n \pars{\norm{v_{k-1}}^2 - \norm{v_k}^2} = \norm{f}^2 -\norm{v_n}^2.
\end{equation*}
\end{proof}
\end{theorem}

We now extend the class of functionals for which convergence of the residual can be shown by modifying the proof of Theorem 2.8 in \cite{VeseDeblurring}. In particular, we focus on the case of, possibly different, seminorm penalties. For this, let us recall the characterization of the subgradient of seminorms. For any seminorm $J$ and any $x_0 \in \dom J$, it is (see for instance Theorem 2.4.14 in~\cite{zalinescu2002convex}):\begin{equation}\label{eq:Subdifferential_seminorm}
    \partial J(x_0) = \set{x^* \in X: \inner{x^*}{x_0} = J(x_0), \ \inner{x^*}{x}\le J(x)\text{ for all }x \in X}. 
\end{equation}

\begin{remark}
In the case of the original MHDM, where $J_n = \inv{(\lambda_n)} J$ with a seminorm $J$, we obtain that $\lambda_n \zeta_n \in \partial J(u_n)$, which by \eqref{eq:Subdifferential_seminorm} implies $\inner{\zeta_n}{u_n} = \frac{J(u_n)}{\lambda_n}$. Thus, \eqref{eq:decomposition} reads as \begin{equation}\label{eq:Decomposition_seminorms}
    \norm{f}^2 =  \norm{v_n}^2 +\sum\limits_{k  = 0}^n\pars{\norm{Tu_k}^2 + 2\frac{J(u_k)}{\lambda_k}}.
\end{equation}
\end{remark}
In particular, \eqref{eq:decomposition} means \begin{equation*}
    \sum\limits_{k = 0}^\infty\pars{\norm{Tu_k}^2 + 2\inner{\zeta_k}{u_k}}\le \norm{f}^2,
\end{equation*}
with equality if and only if $\norm{v_n}$ converges to $0$. Theorem \ref{eq:Convergence_rate_normpowers} and Theorem 2.1 in \cite{Nachman} give sufficient conditions on the choice of parameters for this convergence to happen.

\begin{lemma}\label{lemma:Convergence_Residual_Generalization}
Let $(J_n)_{n \in \N_0}$ be a sequence of seminorms and assume there are constants $C_{l,j}\ge 0$ such that \begin{equation*}
    J_l(u) \le C_{l, j} J_j(u)
\end{equation*} for all $u \in X$ and $l,j \in \N_0$ with $l>j$. Furthermore, assume $\lim\limits_{l \to \infty}C_{l,j} = 0$ for all $j \in \N_0$. If $J_0(x^\dagger) < \infty$, then the sequence of residuals 
$(v_n)$ with 
$v_n = f -T x_n$ defined by the generalized MHDM method 
\eqref{step flexible}
 converges to $0$ in the strong topology of $X$.
\begin{proof}
Fix $k \in \N_0$ and let $N \in \N$. One has \begin{equation*}
    \norm{v_{k+N}}^2 = \inner{v_{k+N}}{v_{k+N}} = \inner{v_{k+N}}{v_k} - \inner{v_{k+N}}{\sum\limits_{i = k+1}^{k+N}Tu_i}.
\end{equation*}
We will show that both summands on the right-hand side of the former equality converge to $0$. For the first one, recall that $T^*v_{i} \in \partial J_{i}(u_{i})$ by \eqref{eq:zeta}, which with \eqref{eq:Subdifferential_seminorm} implies $\inner{T^*v_{i}}{y} \le J_{i} (y)$ for all $y \in X$, with equality if and only if $y = u_i$. Additionally, let $\eps >0$. Using the triangle inequality of $J_{k+N}$, it is \begin{align*}
    \inner{v_{k+N}}{v_k}&= \inner{v_{k+N}}{f-Tx_k} = \inner{T^*v_{k+N}}{x^\dagger-\sum\limits_{i = 0}^ku_i} \le J_{k+N}\pars{x^\dagger - \sum\limits_{i = 0}^ku_i} \\& \le J_{k+N}(x^\dagger) +\sum\limits_{i = 0}^k J_{k+N}(u_i) \le C_{k+N,0}J_0(x^\dagger) + \sum_{i = 0}^k C_{k+N,i}J_i(u_i) \\& = C_{k+N,0}J_0(x^\dagger) + \sum\limits_{i = 0}^kC_{k+N,i}\inner{T^*v_i}{u_i} \le C_{k+N,0} J_0(x^\dagger) + \frac{\eps}{2}\norm{f}^2,
\end{align*} where the last inequality follows from \eqref{eq:decomposition} if $N$ is chosen large enough such that $C_{k+N,i} \le \eps$ for all $i \le k$.  Letting $N \to \infty$ now implies convergence of the first summand. Using the triangle inequality and the subdifferential property of $T^*v_{k+N}$ again, we estimate \begin{align*}
    \abs{\inner{v_{k+N}}{\sum\limits_{i = k+1}^{k+N}Tu_i}} = \abs{\inner{T^*v_{k+N}}{\sum\limits_{i = k+1}^{k+N}u_i}} \le \sum\limits_{i = k+1}^{k+N}J_{k+N}(u_i) \le \sum\limits_{i = k+1}^{k+N} C_{k+N,k+1}J_i(u_i).
\end{align*}
Letting $N \to \infty$ and taking the boundedness of $\sum\limits_{k = 0}^\infty J_k(u_k)$ by \eqref{eq:decomposition} into consideration, we obtain that the second summand converges to $0$, too. 
\end{proof}
\end{lemma}

\begin{remark}
In the situation of the original MHDM, it is $J_l = \frac{1}{\lambda_l} J$, which yields $J_l = \frac{\lambda_j}{\lambda_l}J_j$. Thus, the assumptions of Lemma \ref{lemma:Convergence_Residual_Generalization} are satisfied if and only if $\lim\limits_{j \to \infty}\lambda_j = \infty$ and $J(x^\dagger) < \infty$.
\end{remark}

In the case of arbitrary convex penalty terms $J_n$, we can also adapt part $(ii)$ of Theorem \ref{thm:Convergence_results}.
\begin{lemma}
Let $u_n$ be obtained by \eqref{step flexible} with a sequence of proper, convex, lower-semicontinuous functionals $(J_n)_{n\in \N}$. Assume furthermore that there is $C\in \R$ such that $\set{u : \limsup\limits_{n\to\infty}J_n(u) <C}$ is dense in $X$. If the sequence $(J_n)_{n \in \N_0}$ is uniformly bounded from below and $\sum\limits_{n = 0}^\infty J_n(0) < \infty$, then $(f-Tx_n)_{n \in \N_0}$ is bounded, and every weak limit point is in the kernel of $T^*$. In particular, this implies that $T^*(f-Tx_n)$ converges to $0$ in the weak-*-topology of $X^*$. \begin{proof}
By the same reasoning as in the proof of Theorem \ref{thm:Convergence_results}, we obtain \begin{equation*}
    \norm{f-Tx_n}^2 \le \norm{f}^2 + \sum_{k = 0}^n J_k(0)
\end{equation*}
and \begin{equation*}
    \inner{T^*(f-Tx_n)}{z} \le J_n(z) +K
\end{equation*}
for some constant $K \ge 0$ and all $z \in X$. Passing to weak limit points of $f-Tx_n$ proves the claim in the same way as in part $(ii)$ of Theorem \ref{thm:Convergence_results}.
\end{proof}
\end{lemma}

Let us now illustrate how the decomposition result from Theorem \ref{thm:Decompostion_finite_equality} can be applied to related iterative methods.

\subsection{The tight MHDM}
A tight version of the MHDM was introduced in \cite{Nachman}  in order to ensure boundedness, and consequently convergence of the iterates $x_n$:
\begin{equation}\label{eq:tight_MHDM}
u_n  = \argmin_{u\in X}    \frac{\lambda_n}{2}\norm{T(u+x_{n-1})-f}^2 +\lambda_na_n J(u+x_{n-1}) + J(u),\quad n\in\N_0,
\end{equation} 
with $u_n$ playing the same role as above, i.e.,  $u_n = x_{n}-x_{n-1}$.
 This tight iteration was generalized in \cite{MultiscaleRefinementImaging} to a refined version:
\begin{equation}\label{eq:tight_refinement}
u_n  = \argmin_{u\in X}    \frac{\lambda_n}{2}\norm{T(u+x_{n-1})-f}^2 +\lambda_na_n J(u+x_{n-1}) + R_n(u), \quad n\in\N_0,
\end{equation} 
where $R_n$ are seminorms. 
 Under mild conditions (cf.~\cite{Nachman,MultiscaleRefinementImaging}), convergence of the iterates could be proved up to subsequences for both the tight MHDM and its refinement.

Clearly, \eqref{eq:tight_refinement} is a special case of the general iteration 
\eqref{step flexible} with $J_{n}$ defined as 
\[ J_n(u) =  \lambda_n a_nJ(u+x_{n-1}) +   R_n(u). \]
By Lemma~\ref{thm:Decompostion_finite_equality}, we obtain 
the following decomposition:
\[ \|f\|^2 = \sum_{k=0}^\infty \left(\|T x_{k}-Tx_{k-1}\|^2 +
2 \langle \lambda_k a_k \partial J(x_{k}) +  \partial R_k(u_k), x_{k}-x_{k-1} \rangle\right),
\]
assuming that the conditions for convergence in 
\cite{MultiscaleRefinementImaging} hold. Here $\partial J(x_{k})$ and $\partial R_k(u_k)$ are generic notations for the appropriate subgradients.

\subsection{Bregman iteration}
The well known Bregman iteration can also be considered in 
the framework of generalized penalty terms \eqref{step flexible}.
Recall that the Bregman iteration  is  defined via \eqref{eq:bregman}, i.e., \[ x_n \in \argmin_{x\in X} \frac{\lambda_n}{2}  \|T x -f\|^2 +
D_J^{p_{n-1}}(x,x_{n-1}),\quad n\in\N, \] 
where for the initial step we set $p_0 = 0$ and afterward choose 
\[ p_{n} = \lambda_n T^*(f-Tx_n) +p_{n-1} \in \partial J(x_{n}).\] 
 Substituting $u =x- x_{n-1}$ in the definition of $x_n$ and omitting those terms in the definition of the Bregman distance \eqref{eq:Bregman_Distance} which are independent of $u$, we observe that $u_n: = x_n-x_{n-1}$ is a minimizer of 
\eqref{step flexible} with 
\[ J_n(u) :=  \frac{1}{\lambda_n} \pars{ J(x_{n-1} + u) - \inner{p_{n-1}}{u}}, \qquad p_{n-1} \in \partial J(x_{n-1}).\]
Hence, we may apply the decomposition result \eqref{eq:decomposition} to the Bregman iteration. 
Due to \begin{equation*}
    \frac{1}{\lambda_n}(p_n -p_{n-1}) \in \partial J(x_n) - \set{p_{n-1}} \in \partial J_n(x_n),
\end{equation*}
we obtain  \[ \|f\|^2 = \norm{Tx_n-f}^2 + 
 \sum_{k=0}^n \left(\|T u_k\|^2 + 2
\lambda_k^{-1} D_J^{sym}(x_{k},x_{k-1})\right),  \]
with the symmetric Bregman distance 
\[ D_J^{sym}(x_{k},x_{k-1})  = \langle p_k - p_{k-1}, x_{k}-x_{k-1} \rangle. \]
Since the residual $Tx_n-f$ converges to $0$ as $n \to \infty$ for appropriate parameters $\lambda_n$ (cf. \cite{frick_scherzer, orig_breg}), we get the full decomposition\begin{equation*}
    \|f\|^2 =  \sum_{k=0}^\infty \left(\|T x_{k}-Tx_{k-1}\|^2 + 2
\lambda_k^{-1} D_J^{sym}(x_{k}, x_{k-1})\right).
\end{equation*}
As recalled in the introduction, 
in the case of $J(u) =\frac{\|u\|^2}{2}$ in Hilbert spaces, the Bregman iteration becomes the iterated Tikhonov 
regularization \eqref{eq:nonstationary_tikhonov},

 for which the symmetric Bregman distance $D_J^{sym}(x_{k},x_{k-1})$ reads $\|x_{k}-x_{k-1}\|^2$. 
Thus, the multiscale decomposition for iterated 
Tikhonov regularization in Hilbert spaces has the form
\begin{equation*}
    \|f\|^2 =  \sum_{k=0}^\infty \left(\|T x_{k}-T x_{k-1}\|^2 + 2
\lambda_k^{-1} \|x_{k}-x_{k-1}\|^2\right).
\end{equation*}
To the best of our knowledge, this result does not seem to be known. 

\section{Comparison of the MHDM and the generalized Tikhonov regularization}\label{sect:ComparisonTikhonov}
We will now focus on comparing the iterative MHDM to a single step  Tikhonov regularization,
i.e., classical generalized 
Tikhonov regularization as defined in~\eqref{eq:tikhonov}. Generally, the first iterate $x_0$ of the MHDM is by definition the Tikhonov regularizer at scale $\lambda_0$. Nonetheless, we do not expect the  MHDM iterate $x_k$ to coincide with the solution $x_{\lambda_k}$ of the Tikhonov regularization corresponding to the parameter $\lambda_k$. Yet, there are frameworks in which this unexpected situation occurs anyway. As an introductory example, consider the case of sparse denoising.

\begin{example}\label{ex:sparsity}
Let $X = \ell^2$, $T  =Id$ and $J = \ell^1$. Using the substitution $x = u +x_{n-1}$ as in \eqref{eq:MHDM_step} and setting $x_{-1} = 0$, each iteration step means to compute \begin{equation*}
    x_n = \argmin\limits_{x \in \ell^2}\set{ \frac{\lambda_n}{2}\norm{x-f}_{\ell^2}^2 + \norm{x-x_{n-1}}_{\ell^1}}.
\end{equation*} 
Note that  the $i$-th component of the iterate $x_n$ verifies
\begin{equation*}
    x_n^i = \argmin\limits_{s \in \R}\set{\frac{\lambda_n}{2}(s-f^i)^2 + \abs{s-x_{n-1}^i}}.
\end{equation*}
Therefore \begin{equation}\label{eq:soft_shrinkage_compinents}
    x_n^i = \begin{cases}
    f^i-\frac{1}{\lambda_n} & \text{ if }f^i>x_{n-1}^i+\frac{1}{\lambda_n}\\
    f^i+\frac{1}{\lambda_n} & \text{ if }f^i<x_{n-1}^i-\frac{1}{\lambda_n}\\
     x_{n-1}^i & \text{ if } \abs{f^i-x_{n-1}^i} \le\frac{1}{\lambda_n}
    \end{cases}.
\end{equation}

 Now we can analyze the sequence generated by the MHDM under the assumption that $(\lambda_n)_{n\in \N_0}$ is strictly increasing. We show by induction that  for any component $i$, one has 
\begin{equation}\label{eq:soft_shrinkage}
        x_n^i = \begin{cases}
    f^i-\frac{1}{\lambda_n} & \text{ if }f^i>+\frac{1}{\lambda_n}\\
    f^i+\frac{1}{\lambda_n} & \text{ if }f^i<-\frac{1}{\lambda_n}\\
    0 & \text{ if } \abs{f^i} \le\frac{1}{\lambda_n}
    \end{cases},
\end{equation}
meaning that performing the first $(n+1)$ iterations of the MHDM is nothing but applying the soft shrinkage operator at scale $\frac{1}{\lambda_n}$.  In other words, the $(n+1)$-th step of the MHDM procedure is the same as the convex $\ell^2$ regularization with parameter $\lambda_n$ and $\ell^1$-penalty term, 
\begin{equation*}
    x_n = \argmin\limits_{x \in \ell^2}\set{ \frac{\lambda_n}{2}\norm{x-f}_{\ell^2}^2 + \norm{x}_{\ell^1}}.
\end{equation*}

Indeed, for $n = 0$ this is true by \eqref{eq:soft_shrinkage_compinents}, since by definition $x_{-1} = 0$. Now, assume that \eqref{eq:soft_shrinkage} holds for some $n\in \N$. We distinguish three cases: \begin{enumerate}
    \item Assume $f^{i} > \frac{1}{\lambda_{n+1}}$. If additionally $f^{i}> \frac{1}{\lambda_n}$, we must have $x_n^{i} = f^{i} -\frac{1}{\lambda_n}$ by assumption and therefore, it is $f^{i} > f^{i} -\frac{1}{\lambda_n} + \frac{1}{\lambda_{n+1}} = x_n^{i} +\frac{1}{\lambda_{n+1}}$, so that \eqref{eq:soft_shrinkage_compinents} with $n$ replaced by $n+1$ implies $x_{n+1}^{i} = f^{i} -\frac{1}{\lambda_{n+1}}$. Otherwise, it must be $\frac{1}{\lambda_{n+1}}< f^{i} \le \frac{1}{\lambda_n}$. In that case we have $x_n^{i} = 0$ and hence $f^{i} > x_n^{i}+\frac{1}{\lambda_{n+1}}$. Again, \eqref{eq:soft_shrinkage_compinents} with $n$ replaced by $n+1$ yields $x_{n+1}^{i} = f^{i} -\frac{1}{\lambda_{n+1}}$. 
    \item If $f^{i} < -\frac{1}{\lambda_{n+1}}$, we obtain analogously to the previous case that  $x_{n+1}^{i} = f^{i} + \frac{1}{\lambda_{n+1}}$.
    \item For $\abs{f^{i}} \le \frac{1}{\lambda_{n+1}}$, we must also have $\abs{f^{i}} \le \frac{1}{\lambda_{n}}$ by the monotonicity of the $\lambda_n$. Therefore, one has $x_n^{i} = 0$ by assumption, and $\abs{f^{i} -x_n^{i}} = \abs{f^{i}} \le \frac{1}{\lambda_{n+1}}$. Thus, \eqref{eq:soft_shrinkage_compinents} with $n$ replaced by $n+1$ yields $x_{n+1}^{i} = 0$.
\end{enumerate}

\end{example}

We will now characterize under which conditions the $(n+1)$-th step of the MHDM and the  Tikhonov iteration with parameter $\lambda_n$ coincide. For the remainder of the section, let $J$ be a seminorm. In our further analysis, we consider a dual seminorm which will help  to characterize minimizers of the  Tikhonov functional \eqref{eq:Tikhonov_functional}. In  the special case where $J = \tv{\cdot}$, the results on the dual norm can be found in \cite{Meyer} and \cite{VeseBook}. The reader is referred to  Section 1.3 of \cite{andreu2004parabolic} for the general case of a seminorm $J$.

\begin{definition}
The map\begin{equation*}
    \abs{\cdot}_*: X^* \to \R \cup \set{\infty}, \, \abs{x^*}_* = \sup\limits_{J(x) \neq 0} \inner{x^*}{\frac{x}{J(x)}}
\end{equation*}
 is called the dual seminorm of $X$ induced by the seminorm $J$, where by convention  $\frac{x}{J({x})} = 0$ if $J(x) = \infty$. 
\end{definition}
It can be easily seen that $\abs{\cdot}_*$ is indeed a seminorm. 

\begin{remark}\label{rem:star_norm_subdifferential}
The characterization \eqref{eq:Subdifferential_seminorm} implies that  $\abs{x^*}_* =1$, for any $x_0$ and any subgradient $x^*\in \partial J(x_0)$.
\end{remark}
Using the same arguments as in section 2.1 of \cite{VeseDeblurring}, we can characterize the minimizers of the  Tikhonov functional.

\begin{lemma}\label{lemma:Characterization_minimizer}
Let $J$ be a seminorm. Let $f \in X$ and $\lambda>0$. The following statements hold true:
\begin{enumerate}
    \item $x_\lambda = 0$ is a minimizer of \eqref{eq:Tikhonov_functional} if and only if $\abs{T^*f}_* \le \frac{1}{\lambda}.$
    \item If $\frac{1}{\lambda}<\abs{T^*f}_*<\infty$, then $x_\lambda$ minimizes \eqref{eq:Tikhonov_functional} if and only if
    
    $\abs{T^*(f-Tx_\lambda)}_* = \frac{1}{\lambda}$ and $\inner{T^*(f-Tx_\lambda)}{x_\lambda} = \frac{1}{\lambda}J(x_\lambda)$.
\end{enumerate}
\end{lemma}

We can now use Lemma \ref{lemma:Characterization_minimizer} to analyze when the Tikhonov regularization agrees with the MHDM. Let us recall some notation. For $k \in \N_0$, denote by $x_{\lambda_k}$ a minimizer of the Tikhonov functional with parameter $\lambda_k$ cf. \eqref{eq:tikhonov}.

A subgradient of $J$ at $x_{\lambda_k}$ is given by \begin{equation}\label{eq:xi_lambda}
    \xi_{\lambda_k} := \lambda_k T^*(f-Tx_{\lambda_k})\in \partial J(x_{\lambda_k}).
\end{equation}
Analogously, note that the MHDM iterate $x_k = x_{k-1} +u_k$ with $u_k$ as computed in \eqref{eq:MHDM_initial_step} can be equivalently  obtained as \begin{equation}\label{MHDM with substitution}
    x_k \in \argmin_{x \in X} \frac{\lambda_k}{2} \norm{Tx-f}^2 + J(x-x_{k-1}),
\end{equation}
with 
\begin{equation}\label{deq:xi_k}
    \xi_k := \lambda_k T^*(f-Tx_k) \in \partial J(x_k-x_{k-1}).
\end{equation}

In general, a comparison of $x_{\lambda_k}$ and $x_k$ can be done by using the dual seminorm. For instance, in the case of TV-denoising it was pointed out in Section 3.2 of \cite{VeseMultiscale} that $\norm{x_{\lambda_k} -x_k}_{W^{-1,\infty}} \le \frac{1}{\lambda_k}$. We present next the general form of this result.

\begin{corollary}
Let $J$ be a seminorm and let $x_k$ be obtained as in \eqref{eq:MHDM_initial_step} and \eqref{eq:MHDM_step}. Then, for any $k \in \N_0$, it holds \begin{equation}\label{eq:Tikhonov_estimate_star_norm}
    \abs{T^*T(x_{\lambda_k} -x_k)}_* \le \frac{2}{\lambda_k}.
\end{equation}
\begin{proof}
By \eqref{eq:xi_lambda}, \eqref{deq:xi_k}, \eqref{eq:Subdifferential_seminorm} and the definition of $\abs{\cdot}_*$,  one has \begin{equation*}
    \abs{T^*(f-Tx_k)}_* = \abs{\frac{\xi_k}{\lambda_k}}_* = \frac{1}{\lambda_k}
    \qquad \text{and } \qquad 
    \abs{T^*(f-Tx_{\lambda_k})}_* = \abs{\frac{\xi_{\lambda_k}}{\lambda_k}}_* = \frac{1}{\lambda_k}.
\end{equation*}
Thus, \eqref{eq:Tikhonov_estimate_star_norm} follows by the triangle inequality.
\end{proof}
\end{corollary}

Let us return to the question of when the MHDM iterates  agree with the generalized Tikhonov regularization. Since  the first iterate $x_0$ is obtained via a Tikhonov regularization with parameter $\lambda_0$, the base case for an inductive proof holds. 

The next theorem verifies the induction
step $k\to k+1$. 
Note that the minimizers of Tikhonov regularization problems \eqref{eq:Tikhonov_functional} are not necessarily unique. Therefore, the equality $x_{\lambda_k} = x_k$ should be understood as choosing the same minimizer for both MHDM and Tikhonov minimization problem. 

\begin{theorem}\label{thm:induction_step}
Let $J$ be a seminorm. Let $(\lambda_k)_{k \in \N}$ be an increasing sequence of positive parameters. Fix $k \in \N$ and assume that 
$x_{\lambda_k} =x_k$. 
Then the solution  $x_{\lambda_{k+1}}$ of the Tikhonov regularization with parameter $\lambda_{k+1}$ minimizes the same functional as the MHDM  iterate $x_{k+1}$, i.e.
\begin{equation}\label{eq:Characterization_equivalence1}
    x_{\lambda_{k+1}}\in \argmin\limits_{x \in X} \set{\frac{\lambda_{k+1}}{2}\norm{Tx-f}^2 + J(x-x_k)},
\end{equation} if and only if \begin{equation}\label{eq:cond_1}
    D_J^{\xi_{\lambda_{k+1}}}(x_{\lambda_{k+1}}-x_{\lambda_k}, x_{\lambda_{k+1}}) = 0.
\end{equation} 
\begin{proof}
First assume that $x_{k+1}$ and $ x_{\lambda_{k+1}}$ coincide, i.e., \eqref{eq:Characterization_equivalence1} holds. 
This implies $\xi_{\lambda_{k+1}}\in \partial J(x_{\lambda_{k+1}}-x_{\lambda_k})$. We therefore obtain \begin{align*}
    0 &\le D_J^{\xi_{k+1}}(x_{\lambda_{k+1}}-x_{\lambda_k},x_{\lambda_{k+1}}) = J(x_{\lambda_{k+1}}-x_{\lambda_k}) - J(x_{\lambda_{k+1}}) +\inner{\xi_{\lambda_{k+1}}}{x_{\lambda_{k}}} \\&= -\pars{J(x_{\lambda_{k+1}}) -J(x_{\lambda_{k+1}}-x_{\lambda_k})-\inner{\xi_{\lambda_{k+1}}}{x_{\lambda_{k}}}} = -D_J^{\xi_{\lambda_{k+1}}}(x_{\lambda_{k+1}},x_{\lambda_{k+1}}-x_{\lambda_k})\le 0 .
\end{align*}
Conversely, assume $ D_J^{\xi_{\lambda_{k+1}}}(x_{\lambda_{k+1}}-x_{\lambda_k}, x_{\lambda_{k+1}}) = 0$. We show that $\xi_{\lambda_{k+1}}$ satisfies the optimality conditions of the MHDM, that is $\xi_{\lambda_{k+1}}\in \partial J(x_{\lambda_{k+1}} -x_{\lambda_k})$.  Since $\xi_{\lambda_{k+1}}\in \partial J(x_{\lambda_{k+1}})$, it is by \eqref{eq:Subdifferential_seminorm}  $\abs{\xi_{\lambda_{k+1}}}_* = 1$ and we only need to show $\inner{\xi_{\lambda_{k+1}}}{x_{\lambda_{k+1}}-x_{\lambda_k}} = J(x_{\lambda_{k+1}} - x_{\lambda_k})$. Indeed, since $x_{\lambda_{k+1}}$ minimizes the Tikhonov functional, we have $\inner{\xi_{\lambda_{k+1}}}{x_{\lambda_{k+1}}} = J(x_{\lambda_{k+1}})$. Thus, by \eqref{eq:cond_1} it is \begin{align*}
   \inner{\xi_{\lambda_{k+1}}}{x_{\lambda_{k+1}}-x_{\lambda_k}} &= \inner{\xi_{\lambda_{k+1}}}{x_{\lambda_{k+1}}-x_{\lambda_k}} + D_J^{\xi_{\lambda_{k+1}}}(x_{\lambda_{k+1}}-x_{\lambda_k}, x_{\lambda_{k+1}}) \\&= \inner{\xi_{\lambda_{k+1}}}{x_{\lambda_{k+1}}-x_{\lambda_k}} + J(x_{\lambda_{k+1}}-x_{\lambda_k}) - J(x_{\lambda_{k+1}}) +\inner{\xi_{\lambda_{k+1}}}{x_{\lambda_{k}}}\\& = \inner{\xi_{\lambda_{k+1}}}{x_{\lambda_{k+1}}} +J(x_{\lambda_{k+1}}-x_{\lambda_k}) - J(x_{\lambda_{k+1}}) \\&= J(x_{\lambda_{k+1}}-x_{\lambda_k}).
\end{align*}   
\end{proof}
\end{theorem}

\begin{remark}
If the minimizer of the Tikhonov functional \eqref{eq:Tikhonov_functional} is unique, then Theorem \ref{thm:induction_step} states that $x_{\lambda_k} = x_k$ for all $k \in \N_0$ if and only if $D_J^{\xi_{k+1}}(x_{\lambda_{k+1}} -x_{\lambda_k},x_{\lambda_k}) = 0$ for all $k \in \N_0$. 

\end{remark}

By symmetry, we can also characterize when the iterates of the MHDM minimize the corresponding Tikhonov functional.

\begin{corollary}
Let $(\lambda_k)_{k \in \N}$ be an increasing sequence of positive parameters. Fix $k \in \N$ and assume that $x_k$ minimizes the Tikhonov functional with parameter $\lambda_k$. Then $x_{k+1}$ minimizes the Tikhonov functional with parameter $\lambda_{k+1}$ if and only if \begin{equation}
    D_J^{\xi_{k+1}}(u_{k+1},x_{k+1}) = 0.
\end{equation}
\begin{proof}
The proof follows analogously to the one in Theorem \ref{thm:induction_step}.
\end{proof}
\end{corollary}

It is remarkable that we can now see whether the iterates of the MHDM can also be obtained via Tikhonov regularization just by knowing the Tikhonov minimizers. We will thus derive an equivalent formulation of \eqref{eq:cond_1}, which is easily verifiable knowing the Tikhonov minimizers. For this, we need to characterize the intersection of subdifferentials of seminorms.

\begin{proposition}\label{prop:Intersection_subdifferentials}
Let $X$ be a normed space and $J$ be a seminorm on $X$. For any $z_1,z_2\in X$ and $z^* \in X^*$, the following equivalence holds:\begin{equation*}
    z^* \in  \partial J(z_1) \cap \partial J(z_2)\text{ if and only if } z^* \in \partial J(z_1+z_2) \text{ and } J(z_1) +J(z_2) = J(z_1+z_2).
\end{equation*}  
\begin{proof}
First assume $z^* \in \partial J(z_1) \cap \partial J(z_2)$. By \eqref{eq:Subdifferential_seminorm} this means $\inner{z^*}{z} \le J(z)$ for all $z \in X$, $\inner{z^*}{z_1} = J(z_1)$ and $\inner{z^*}{z_2} = J(z_2)$. Therefore by the triangle inequality of $J$ and again \eqref{eq:Subdifferential_seminorm}, one has
\begin{equation*}
    J(z_1+z_2)\le J(z_1) +J(z_2) = \inner{z^*}{z_1+z_2} \le J(z_1+z_2). 
\end{equation*}
This implies $\inner{z^*}{z_1+z_2} = J(z_1+z_2)$, that is $z^* \in \partial J(z_1+z_2)$. Furthermore, we obtain $J(z_1) + J(z_2) = J(z_1+z_2)$. \\
Conversely, assume $z^* \in \partial J(z_1+z_2)$ and $J(z_1) +J(z_2) = J(z_1+z_2)$. By \eqref{eq:Subdifferential_seminorm}, we immediately get $\inner{z^*}{z}\le J(z)$ for all $z \in X$ and \begin{equation*}
  \inner{z^*}{z_1+z_2} = J(z_1+z_2) = J(z_1) + J(z_2). 
\end{equation*} 
This yields, \begin{align*}
    J(z_1) &= J(z_1+z_2) -J(z_2) = \inner{z^*}{z_1+z_2} -J(z_2) \le \inner{z^*}{z_1+z_2} - \inner{z^*}{z_2} = \inner{z^*}{z_1}\le J(z_1).
\end{align*}
Hence, it is $\inner{z^*}{z_1} = J(z_1)$ and thus $z^* \in \partial J(z_1)$. By analogous reasoning we get $z^* \in \partial J(z_2)$.
\end{proof}
\end{proposition}

\begin{lemma}
The condition for agreement 
of Tikhonov and MHDM, condition \eqref{eq:cond_1}, is equivalent to \begin{equation}\label{eq:condition_2}
    J({x_{\lambda_{k+1}}-x_{\lambda_{k}}}\ ) + J(x_{\lambda_{k+1}}) = J({2x_{\lambda_{k+1}}-x_{\lambda_{k}}})
\end{equation}
\begin{proof}
We have seen in the proof of Theorem \ref{thm:induction_step} that condition
\eqref{eq:cond_1} is equivalent to \begin{equation*}
    \xi_{\lambda_{k+1}}\in \partial J(x_{\lambda_{k+1}}-x_{\lambda_k})\cap \partial J(x_{\lambda_{k+1}}).
\end{equation*} Then applying Proposition \ref{prop:Intersection_subdifferentials} proves the claim.
\end{proof}
\end{lemma}

Let us now verify some examples using \eqref{eq:condition_2}.

\begin{example}\label{ex:verification_sparse_denoising}
We revisit example \ref{ex:sparsity}. That is $X = \ell^2$, $J = \norm{\cdot}_{\ell^1}$ and $T = Id$. Let now $(\lambda_k)_{k \in \N_0}$ be an increasing sequence of positive numbers. The regularizers $x_{\lambda_k}$  coincide with the iterates of the MHDM as defined in \eqref{eq:soft_shrinkage}. It suffices to verify \eqref{eq:condition_2} componentwise. Let $k,i \in \N$. We distinguish $3$ cases:\begin{enumerate}
    \item Assume $f^i=0$. Then $x^i_j = 0$ for all $j \in \N$ and \eqref{eq:condition_2} clearly holds.
    \item Assume $f^i >0$. If $f^i \le \frac{1}{\lambda_{k+1}}$, we have $x_{k+1}^i = 0$ and thus \eqref{eq:condition_2} holds. Thus assume $f^i > \frac{1}{\lambda_{k+1}}$. Once again we do not need to consider the case $f^i \le \frac{1}{\lambda_k}$. So we assume $f^i \ge \frac{1}{\lambda_k}$. This yields $x^i_{\lambda_k} = f^i -\frac{1}{\lambda_k}>0$ and $x^i_{k+1} = f^i - \frac{1}{\lambda_{k+1}}>0$. We therefore obtain by the monotonicity of $\lambda_k$ that \begin{align*}
        \abs{x^i_{\lambda_{k+1}} - x^i_{\lambda_k}} + \abs{x^i_{\lambda_{k+1}}} = \frac{1}{\lambda_k}-\frac{1}{\lambda_{k+1}} + f^i -\frac{1}{\lambda_{k+1}} = 2\pars{f^i -\frac{1}{\lambda_{k+1}}} -\frac{1}{\lambda_k} = \abs{2x^i_{\lambda_{k+1}} -x^i_{\lambda_k}}.
    \end{align*} 
    \item If $f^i <0$ the claim follows analogously to the previous case.
\end{enumerate}
Hence, we see that \eqref{eq:condition_2} holds for all $k \in \N$. Since the initial step of the MHDM is to compute a Tikhonov minimization with parameter $\lambda_0$ and $T = Id$ (implying that the Tikhonov minimizers are unique), we can conclude again that the iterates of the MHDM coincide with the Tikhonov minimizers at the corresponding parameter. 
\end{example}

The next part deals with an extension of this example to 
the case of Tikhonov regularization with $\ell^1$ penalty for general operators $T$, where  sufficient conditions 
such that the MHDM iteration coincides with Tikhonov regularization are established.

\subsection{\texorpdfstring{$\ell^1$}{TEXT}-regularization in finite dimensions} 
The aim of this subsection is to 
prove that the so-called positive cone condition 
(see \cite{Efron04}) implies \eqref{eq:condition_2} in a finite-dimensional case for $\ell^1$ regularization.
We consider the 
$\ell^1$ Tikhonov regularization
\begin{equation}\label{problem} 
\frac{\lambda}{2} \|Tx - y\|_2^2 +  \|x\|_{1},  \end{equation} 
where $\lambda$ is a positive parameter and $T \in \R^{m\times n}$ is an injective  matrix: $\ker T = \{0\}$. Note that 
this implies that $T^*T$ is invertible. 
We denote by $x_\lambda$ the minimizer of 
\eqref{problem}. In the following, the notation $x^i$ is used 
to refer to the $i$-th component of a vector $x$.
\begin{remark}\label{remark_monotone} One can check that  condition \eqref{eq:condition_2} is verified in case  $J = \|.\|_{1}$
if the inequality 
\begin{equation}\label{ineq4} \abs{x_{\lambda_k}^i} \leq \abs{x_{\lambda_{k+1}}^i} \end{equation}
holds for all components where $x_{\lambda_{k+1}}^i \not =0.$
    
\end{remark}

We formulate the positive cone condition
of \cite{Efron04}, 
stated as diagonal dominance in 
\cite{Duan11}.
\begin{definition}\label{defpc}
 For some index set $J \subset\{1,\ldots, n\}$, we denote by  
 $T^J$ the corresponding submatrix of $T$ that takes only the columns in $J$, that is, $(T^J)_{i,j}= T_{i,j}$ for all $i\in \{1,\ldots m\}$ and $j \in J$.
 Define $S^J$ as the matrix 
 $S_J:= \left((T^J)^*T^J\right)^{-1} \in \R^{|J| \times |J|}$.

 We say that a matrix $T$ satisfies the 
 \emph{positive cone condition} if for all $J \subset\{1,\ldots, m\}$, the 
 matrix $S_J$ is diagonally dominant, i.e., 
 \[ r_{J,i}:= {(S_{J})}_ {i,i} - \sum_{j \not = i} | {(S_{J})}_{i,j}|  \geq 0 \qquad \forall i \in J. \]
\end{definition}
To verify the positive cone condition, the following result from 
\cite[Lemma 4]{Duan11} is useful.

\begin{lemma}\label{lem0}
In the setting of Definition~\ref{defpc}, a matrix $T$ satisfies 
the positive cone condition if and only if $(T^*T)^{-1}$ is 
diagonally dominant, i.e., $S:= (T^*T)^{-1}$ satisfies 
\[ S_{i,i} - \sum_{j\not= i} |S_{i,j}| \geq 0 \qquad \forall i\in\{1,\dots,m\}. \]
\end{lemma}

Before stating the main result of this subsection, we need some lemmas, and here we strongly rely on the 
finite-dimensionality.

For later references we state now the 
optimality conditions for \eqref{problem},
\begin{equation}\label{opc}  \lambda T^*T x_\lambda +  \xi_\lambda = T^*y, \qquad \text{ with }\xi_\lambda \in \partial (\norm{\cdot}_1)(x_\lambda), 
\end{equation}
and $\xi_\lambda^i \in [-1,1]$ with 
$\xi_\lambda^i = \text{sign}(x_\lambda^i)$
whenever $x_\lambda^i \not = 0$. 

We start with a lemma that 
establishes continuity of the solution $x_\lambda$ 
with respect to $\lambda$ in the 
finite-dimensional setup 
(cf.~similar results  in 
\cite{Efron04} under the strict positive 
cone condition).  
\begin{lemma}\label{lem1} 
Let $\lambda >0$. The mapping  
\[ \lambda \mapsto x_{\lambda}  \]
is continuous from $\R^+$ into $ \R^n$. 
\end{lemma}
\begin{proof}
Let $\lambda, \mu \in \R^+$ and subtract \eqref{opc} for $\lambda$ and $\mu$ from each other. This yields
\[ T^*T (x_\lambda-x_\mu) + 
\lambda^{-1} (\xi_{\lambda} - \xi_\mu) + 
(\lambda^{-1} -\mu^{-1}) \xi_\mu = 0.  \] 
Taking the inner product with $(x_{\lambda} - x_\mu)$ gives   
\[ \|T(x_\lambda-x_\mu)\|^2 +  
\lambda^{-1} \inner{\xi_{\lambda} - \xi_\mu}{x_{\lambda} - x_\mu} =
-(\lambda^{-1} -\mu^{-1}) \inner{\xi_\mu}{x_{\lambda} - x_\mu}. 
\]
By convexity, the second term on the left-hand side is nonnegative, while by injectivity and 
finite-dimensionality the first term 
has a lower bound $c \|x_{\lambda} - x_\mu\|^2$,
where $c$ is the smallest singular value of 
$T$. Thus, 
\[ c \|x_{\lambda} - x_\mu\|^2 \leq  
-(\lambda^{-1} -\mu^{-1}) \inner{\xi_\mu}{x_{\lambda} - x_\mu} \leq  
|\lambda^{-1} -\mu^{-1}| \|x_{\lambda} - x_\mu\|,
\]
which yields continuity. 
\end{proof}

For the next results we need 
the following index sets:
\begin{align*}  
N(\lambda) & =  \{i \in \{1,\ldots, n\}\,:\,x_\lambda^i=0 \}  \\ 
I(\lambda) & 
= \{i \in \{1,\ldots, n\} \,:\, x_\lambda^i \not =0 \} 
\end{align*}

\begin{lemma}\label{lem2}
Let $\mu < \lambda$ and suppose that $T$ satisfies the positive cone condition. 
Furthermore, suppose that the following 
conditions hold:
\begin{align}\label{null} N(\lambda) &\subset N(\mu), \\
 \xi_\lambda^i &= \xi_\mu^i \qquad \text{for all }  i \in I(\lambda)    \label{null1}
\end{align}
Then 
\begin{equation}\label{ineq} (x_\mu^i) \sign(x_\lambda^i) \leq \abs{x_\lambda^i} \qquad \text{ for all } i \in I(\lambda). \end{equation}
\end{lemma}
\begin{proof}
We use a partitioning of the vectors into the index sets  $I(\lambda)$ and $N(\lambda)$,
\[ x= \begin{bmatrix} x^{I(\lambda)}  \\ x^{N(\lambda)} \end{bmatrix}, \] 
where the upper part contains the indices of 
$I(\lambda)$ and the lower part those in $N(\lambda)$. 
Let $\lambda,\mu$ satisfy the conditions above. By using that 
$x_\mu^i = 0$ for $i \in N(\lambda)$ , it follows from the optimality conditions that 
\[ T^*T \begin{bmatrix} x_\mu^{I(\lambda)} -x_\lambda^{I(\lambda)} \\ 0 \end{bmatrix} 
+ \begin{bmatrix} (\mu^{-1}   -\lambda^{-1})  \xi_\lambda^{I(\lambda)} \\ 
\mu^{-1} \xi_\mu -\lambda^{-1} \xi_\lambda^{N(\lambda)} \end{bmatrix}  = 0, 
\] where $\xi_\lambda^{I(\lambda)} := (\xi_\lambda^i)_{i \in I(\lambda)}$ and $\xi_\lambda^{N(\lambda)} := (\xi_\lambda^i)_{i \in N(\lambda)}$. Let $i \in I(\lambda)$. Since $\xi_\lambda^i = \sign(x_\lambda^i)$, we obtain 
\begin{align*}  x_\mu^i -x_\lambda^i  &=
-(\mu^{-1}-\lambda^{-1})  \pars{S_{I(\lambda)} \xi_\lambda^{I(\lambda)}}_i  \\
& =
-(\mu^{-1}-\lambda^{-1}) \text{sign}(x_\lambda^i) \left[ (S_{I(\lambda)})_{i,i}   + \sum_{j\in I(\lambda)\setminus\set{i}}   (S_{I(\lambda)})_{i,j} \frac{\text{sign}(x_\lambda^j)}{ \text{sign}(x_\lambda^i) } \right]
\end{align*} 
Because 
\[ (S_{I(\lambda)})_{i,i}   + \sum_{j\in I(\lambda)\setminus\set{i}}    (S_{I(\lambda)})_{i,j} \frac{\text{sign}(x_\lambda^j)}{ \text{sign}(x_\lambda^i) }  \geq 
  (S_{I(\lambda)})_{i,i} -  \sum_{j\in I(\lambda)\setminus\set{i}}   |(S_{I(\lambda)})_{i,j}| = r_{I(\lambda),i} \geq 0, 
\]
it follows that 
\[  (x_\mu^i -x_\lambda^i)\text{sign}(x_\lambda^i)   \leq 0,\] 
which is equivalent to \eqref{ineq}
\end{proof}

\begin{lemma}\label{nulllemma}
Let $\mu< \lambda$ and suppose that $T$ satisfies the positive cone condition. 
Then there exists an $\eps>0$ such that 
\[ N(\lambda) \subset N(\mu) \qquad \forall \mu \in [\lambda -\eps,\lambda] . \]
\end{lemma}
\begin{proof} 
Suppose this is not the case. Then there exist a monotonically increasing sequence $(\mu_k)_{k\in \N}$ with 
$\mu_k<\lambda$ and $\lim\limits_{k\to \infty} \mu_k = \lambda$,
and an index sequence $(i_k)_{k\in \N} \subset N(\lambda)$ 
such that $x_{\mu_k}^{i_k} \not = 0$ while $x_{\lambda}^{i_k} = 0$. By the continuity 
in Lemma~\ref{lem1},
this sequence satisfies $\lim\limits_{k\to \infty} x_{\mu_k}^{i_k}  = 0$. 

We proceed by constructing a subsequence of $x_{\mu_k}$ that 
has the same set of non-zero components for all its elements. Note that by finite dimensionality and continuity (Lemma \ref{lem1}) we may assume $I(\mu_k) \subset I(\lambda)$ for all $k$. Now, pick $j \in N(\lambda)$ such that 
$x_{\mu_k}^{j} \not = 0$ for infinitely many $k$. Such an index $j$ must exist since $(i_k)_{k\in \N}$ is a sequence in a finite set, and thus it has to meet some index $j$ infinitely many times. 
Take a subsequence  (again denoted by the index $k$) such that $x_{\mu_k}^{j} \not = 0$  for all $k\in\N$, then
 set $N^\infty = \{j\}$ and $N^F = N(\lambda)\setminus\{j\}$. We now proceed inductively with this 
construction for $j \in N^F$.  Try to find an index $j \in N^F$ such that $x_{\mu_k}^{j} \not = 0$ for infinitely many $k$. 
Such an  index may or may not exist. If it exists, we again take a subsequence (again denoted by an index $k$)
such that $x_{\mu_k}^{j} \not = 0$ for all $k\in \N$, we add $j$ to $N^\infty$ and remove it from $N^F$. We proceed  with this 
construction until  either no such $j$ can be found (case (i))  or $N^F$ is the empty set (case (ii)) -  
this situation must happen since the index set  is  a finite set. 
In  case (i), we have for all $j \in N_F$,  that $x_{\mu_k}^{j} \not = 0$ for only finitely many $k$. 
Thus, we can take a subsequence (again denoted by an index $k$) such that 
$x_{\mu_k}^{j} = 0$   for all $k$ and all $j \in N^F$.  
By this construction, we find a subsequence and an index set partition $N^F \cup N^\infty = N(\lambda)$ 
(including the case of $N^F$ being empty, case (ii)) such that 
\[ x_{\mu_k}^{j} \not = 0 \quad \text{ for } j \in N^\infty , \qquad 
 x_{\mu_k}^{j} = 0  \text{ for } j \in N^F.  
\]
Thus, we have $I(\mu_k) = I(\lambda) \cup N^\infty$ for all $k$.

We note that the optimality condition and Lemma~\ref{lem1} also imply that $\xi_{\mu}$ is 
continuous in $\mu$. For $j \in N^\infty$, we have that   
$\xi_{\mu_k}^j = \text{sign}(x_{\mu_k}) \in \{-1,1\}$ is continuous in $\mu$ and hence, for $k$ sufficiently large, 
this implies that the sign remains constant for sufficiently large $k \geq n_0$, that is,
$\xi_{\mu_k}^j = \xi_{\mu_{k'}}^j$ for all $k' \geq k$. Again by continuity it follows 
for $i \in I(\lambda)$ that 
$\lim\limits_{k\to \infty} \xi_{\mu_k}^i  = 
\xi_{\lambda}^i \in \{-1,1\},$
and hence 
$\text{sign}(x_{\mu_k}^i) = \xi_{\mu}^i$
must remain constant 
 also on $I(\lambda)$ for $k$ sufficiently large.

It follows that  the constructed sequence 
satisfies the properties \eqref{null} and \eqref{null1}  
for any pair $\mu_k < \mu_{k'}$ 
with $n_0 \leq k\leq k'$. 
Again by continuity we 
also have $\text{sign}(x_{\mu_k}^i) = \text{sign}(x_{\mu_{k'}}^i)$ for $i \in N^\infty$ which implies,
using \eqref{ineq},
\[ |x_{\mu_{k}}^i| = 
x_{\mu_{k}}^i \xi_{\mu_{k}}^i 
= x_{\mu_{k'}}^i \xi_{\mu_{k}}^i 
\leq  |x_{\mu_{k'}}^i|, \qquad  n_0 \leq k\leq k', i \in N^\infty.\]
However this contradicts the condition that $x_{\mu_{k'}}^i \to 0$ for $i \in N^\infty$, hence
the proposition is verified. 
\end{proof}

We prove in the sequel a local monotonicity result.

\begin{proposition}\label{pp} 
Let $T$ satisfy the positive cone condition. 
Then, for any $\lambda$, there exists an $\eps>0$ such that 
\[ |x_\mu^i| \leq |x_\lambda^i| \qquad \forall i \in I(\lambda) , \forall \mu \in [\lambda-\eps,\lambda]\]

\end{proposition}
\begin{proof}
Let $\mu <\lambda$. It follows by 
continuity that $\text{sign}(x_\lambda^i)=\xi_\lambda^i = \xi_\mu^i=\text{sign}(x_\mu^i)$ for $i \in I(\lambda)$ 
and $\abs{\mu-\lambda}$ small enough. 
Taking additionally $\mu$ close enough to 
$\lambda$ so that Lemma~\ref{nulllemma}
applies, we may use Lemma~\ref{lem2} 
to conclude the result
\[ |x_\mu^i| = 
x_\mu^i \text{sign}(x_\mu^i) = 
x_\mu^i \text{sign}(x_\lambda^i) \leq 
|x_\lambda^i|.
\]
\end{proof}

By continuity, we may globalize the result as follows.
 \begin{proposition}
 Let $T$ satisfy the positive cone condition. 
Then, for any $\lambda>0$, we have 
\[ |x_\mu^i| \leq |x_\lambda^i|  \qquad \forall i \in I(\lambda) , \forall \mu \leq \lambda.   \] 

\end{proposition}
\begin{proof}
Suppose that the  statement does not hold. Then for some $i \in I(\lambda)$
there exists $\mu^*<\lambda$ with 
\[  |x_{\mu^*}^i| > |x_{\lambda}^i|. \]  
Take $\bar{\mu}$ as the supremum of all such $\mu^* < \lambda$.  It follows then by continuity 
that $|x_{\bar{\mu}}^i| =  |x_\lambda^i|$,
and there exists a sequence $(\mu_k)_k$ converging to $\bar{\mu}$ with   
$\mu_k< \bar{\mu}$ and
$|x_{\mu_k}^i| > |x_{\bar{\mu}}^i|$, 
which contradicts Proposition~\ref{pp}.
\end{proof} 
We note that monotonicity of components
was also proven by Meinshausen~\cite{Mein07}
under the slightly stronger restricted 
positive cone condition. The result
of the theorem has also been stated 
in \cite[Remark 3]{Duan11} though without 
full proof. 

We are now in a position to state the main result of this subsection, its proof being a consequence of the auxiliary results shown above. 

\begin{theorem}\label{th:ell1}
 Let $T\in\R^{m\times n}$ be injective and assume that $T$ satisfies the positive cone condition. 
 Then the  $\ell^1$-minimizers  of \eqref{problem} satisfy  
  \eqref{eq:condition_2} with $J=\|\cdot\|_{1}$ and with corresponding regularization parameters $\lambda_k$, for any $k\in\N_0$. In particular, in this case 
  the MHDM iteration agrees with the corresponding $\ell^1$-regularization, that is
  \[ x_n = x_{\lambda_n}, \forall n\in\N_0.\]

\end{theorem}

\begin{proof}
    Proposition \ref{pp} applies for $\mu=\lambda_k$ and $\lambda=\lambda_{k+1}$ yielding 
 \eqref{ineq4} and thus, 
\eqref{eq:condition_2}.
\end{proof}

We finish with a denoising example where the MHDM does not agree with Tikhonov regularization. This will be due to a violation of 
\eqref{eq:condition_2}. 

\begin{example}
Let \begin{equation*}
    T = \begin{bmatrix*}[r]
    2 &1 \\1 &0
    \end{bmatrix*}
\end{equation*}
and $f = (4,-1)^T$. The corresponding Tikhonov minimizer \begin{equation*}
    (x_\lambda^1,x_\lambda^2) \in \argmin_{(x_1,x_2)\in \R^2} 
   \frac{\lambda}{2 } \norm{T(x_1,x_2)^T-f}_2^2 + \abs{(x_1,x_2)}_1
\end{equation*} is given by \begin{equation*}
    (x_\lambda^1, x_\lambda^2) = \begin{cases}
    (0,0) &\text{ for } 0 <\lambda\le \frac{1}{7},\\
     \pars{\frac{7}{5}-\frac{1}{5\lambda}, \ 0} &\text{ for } \frac{1}{7}<\lambda\le \frac{1}{2},\\
     \pars{\frac{1}{\lambda}-1, \  6-\frac{3}{\lambda}} & \text{ for } \frac{1}{2}\le \lambda \le 1,\\
     \pars{0, \ 4-\frac{1}{\lambda}} &\text{ for } 1\le \lambda\le 3,\\
     \pars{\frac{3}{\lambda}-1, \ 6-\frac{7}{\lambda}} &\text{ for }\lambda\ge 3.
    \end{cases} 
\end{equation*}
Note that for $\lambda\in \pars{\frac{1}{2},1}$ we have that $x_\lambda^1$ is positive and strictly decreasing to $0$, while $x_\lambda^2$ is positive and strictly increasing with respect to $\lambda$. If the sequence $(\lambda_k)_{k \in \N_0}$ is chosen such that  $\lambda_k, \lambda_{k+1} \in \pars{\frac{1}{2},1}$ and $x_{\lambda_{k+1}}^1 <x_{\lambda_k}^1 <2x_{\lambda_{k+1}}^1 $  hold for some $k \ge 0$, we obtain \begin{equation*}
    \norm{x_{\lambda_{k+1}}-x_{\lambda_k}}_1 + \norm{x_{\lambda_{k+1}}}_1 = \abs{x_{\lambda_k}^1} + 2\abs{x_{\lambda_{k+1}}^2} - \abs{x_{\lambda_k}^2}
\end{equation*} but \begin{equation*}
    \norm{2x_{\lambda_{k+1}} -x_{\lambda_k}}_1 = 2 \abs{x_{\lambda_{k+1}}^1} - \abs{x_{\lambda_k}^1} +2\abs{x_{\lambda_{k+1}}^2} - \abs{x_{\lambda_k}^2}.
\end{equation*}
Therefore, condition \eqref{eq:condition_2} is not satisfied and the MHDM does not agree with Tikhonov minimization. Furthermore, note that \begin{equation*}
    \inv{(T^*T)} = \begin{bmatrix*}[r]
    1 &-2\\-2&5
    \end{bmatrix*}
\end{equation*}
is not diagonally dominant, meaning that $T$ does not satisfy the positive cone condition.
\end{example}

\subsection{TV-denoising in one dimension} 
The results of the previous section can be used to 
analyze also the one-dimensional total variation (TV) denoising problem. The main idea is to first consider the finite dimensional problem. By the use of a substitution, we will transform TV-denoising into a problem of the form \eqref{problem}, which  satisfies the assumptions of Theorem \ref{th:ell1}. Note that a similar approach was used in \cite[Section~IV]{Duan11}.
The main point is that finite dimensional TV denoising is equivalent to TV denoising on the class of piecewise constant functions with jumps at predetermined points. This will allow using approximation arguments to obtain the infinite dimensional case.

In finite dimensional TV regularization, the penalty is 
essentially the $\ell^1$-norm of the derivative. 
We consider now the Tikhonov functional 
\begin{equation}\label{discTVden} \frac{\lambda}{2} \|x-y\|^2_2 + \|D x\|_{1} \end{equation}

where $y,x \in \R^n$, 
and we denote by $x_\lambda$ a minimizer of this functional with respect to $x$. We consider a finite-dimensional situation and 
$D x$ a standard difference quotient, that is, for  $x \in \R^n$, the matrix $D \in \R^{(n-1)\times n}$ has the form \begin{equation}\label{eq:Discrete_derivative}
 D = \begin{bmatrix*} -1 & 1 & 0 & \ldots & 0 \\ 
0 & -1 & 1 & \ldots & 0 \\ 
\ldots & \ldots & \ldots & \ldots & \ldots \\ 
0 & \ldots & 0 &  -1 & 1 
\end{bmatrix*}.
\end{equation}

Note that $D$ has as nullspace $N(D)$,  the subspace of constant vectors. Since we 
can decompose the $\R^n$-space into 
$N(D)\oplus N(D)^\bot$ orthogonally, it is 
not difficult to see that \eqref{discTVden}
can be replaced by  the corresponding optimization problem with $x \in N(D)^\bot$
and $y$ replaced by $P_{N(D)^\bot} y$, i.e., 
the orthogonal projection onto $N(D)^\bot$. 
(The component of $x$ in $N(D)$ is easily 
calculated as the orthogonal projection of $y$
to $N(D)$.) 
Thus, by now considering 
\eqref{discTVden} in $N(D)^\bot = R(D^*)$,
we may set $x = D^* w$ and minimize over
$w \in \R^{n-1}$. Upon setting $z = D D^*w$ and observing that 
$D D^*$ is invertible, we arrive at the function
\[ \frac{\lambda}{2} \|D^*(D D^*)^{-1} z-y\|^2 + \|z\|_{1} \]
to be minimized over $z$. This is the setup of the 
previous section with $T = D^*(D D^*)^{-1}$ and
\[ (T^*T)^{-1} = D D^* = 
\begin{bmatrix*} 2 & -1 & 0 & \ldots & 0 \\ 
-1 & 2 & -1 & \ldots & 0 \\ 
\ldots & \ldots & \ldots & \ldots & \ldots \\ 
0 & \ldots & 0 &  -1 & 2 
\end{bmatrix*},
\]
where the latter operator corresponds to the standard second-order difference quotient. Thus, it
satisfies the positive cone condition according to Lemma~\ref{lem0}. 
As a consequence, the MHDM agrees with 
Tikhonov regularization  in this case  by Theorem \ref{th:ell1}.  Note that this is 
not necessarily true for a general 1D regularization
of the form 
\[ \frac{\lambda}{2} \|T x-y\|^2_2+ \|D x\|_{1}. \]
with  a more general operator $T$.

By a limit argument, we can prove the same result in the continuous case in one dimension.
\begin{theorem}\label{thm:1d_agreeing}
Let  $y \in L^2([0,1])$ and consider the TV-denoising problem in one-dimension
on the interval $(0,1)$, 
i.e., \eqref{eq:Tikhonov_functional} with $T = Id$ and
$J = |.|_{TV}$. 
Then the MHDM iteration agrees with the corresponding 
Tikhonov regularization as in Theorem~\ref{thm:induction_step}
\end{theorem}

\begin{proof}
Due to standard density arguments in $L^2([0,1])$,   $y$ can be approximated by  piecewise constant functions $y_N$ such that 
$\|y-y_N\|_{L^2(0,1)} \to 0$ as $N\to\infty$. To specify the notation, the function $y_N$  is constant on the intervals
$(s_{i},s_{i+1})$, where  $(s_i)_{i=0}^N$ represent the nodes on a uniform grid of $[0,1]$.

According to \cite[Lemma 4.34]{Scherzerbook}, the corresponding minimizer of TV-denoising, 
now denoted by $x_{\lambda_k}^N$, is again piecewise constant on the same grid. One can see that the vectors of coefficients of  $x_{\lambda_k}^N$
on $(s_{i},s_{i+1})$ are the solutions 
of the discrete TV-denoising problems \eqref{discTVden} with corresponding vector of coefficients of $y_N$ as data $y$. 
Moreover, the expression $\|D x\|_{1}$ equals the TV-seminorm  $|x|_{TV}$ for all piecewise constant $x$. Note that the MHDM algorithm iteratively provides solutions of denoising problems, and the solutions to those are piecewise constant on the same grid. Consequently, we inductively obtain that applying the MHDM with data $y_N$ is equivalent to applying the MHDM to a discrete denoising problem with the coefficients of $y_N$ as data.
Hence, the MHDM agrees with Tikhonov regularization (i.e., TV-denoising) 
in this case when $y$ is replaced by $y_N$. We now verify \eqref{eq:cond_1}
by taking limits and using stability of the regularization scheme. For fixed parameter $\lambda_k$, 
$x_{\lambda_k}^N$ depends continuously of $y_N$ in a sense made precise in \cite{HoKaPoSc07};
see also \cite[Th. 2.4]{PoReSc10}.
In fact,  the following hold for $N \to \infty$ and for a subsequence (denoted the same as the original sequence):
\begin{alignat}{2} 
 |x_{\lambda_k}^N|_{TV} &\to |x_{\lambda_k}|_{TV}& \qquad  &\text{by \cite[Theorem 3.2]{HoKaPoSc07}}, \label{help1} \\ 
 x_{\lambda_k}^N & \rightharpoonup^*_{BV} x_{\lambda_k}&   &\text{by \cite[Theorem 3.2]{HoKaPoSc07}},\notag \label{help2} \\
 x_{\lambda_k}^N -y_N &\rightharpoonup_{L^2} x_{\lambda_k} - y& \qquad &\text{see \cite[Proof of Theorem 3.2]{HoKaPoSc07}}, \\ 
 \| x_{\lambda_k}^N -y_N &\| \to \|x_{\lambda_k} - y\|&  &\text{see \cite[Eq. (8)]{HoKaPoSc07}}. 
\end{alignat}
Since from any subsequence one can extract another subsequence with those convergence properties, they must hold for the sequence itself. The last two identities imply strong $L^2$-convergence of  $x_{\lambda_k}^N -y_N$ to  $x_{\lambda_k} -y$ 
by the Radon-Riesz property of $L^2$ (cf.~also \cite[Eq (12)]{PoReSc10}). Based on the convergence of $y_N$ to $y$ in $L^2([0,1])$, we can also conclude the convergence of $x_{\lambda_k}^N$ to $x_{\lambda_k}$. Since by definition it is $\xi_{\lambda_{k}} = y - x_{\lambda_k}$, 
this means that $\xi_{\lambda_{k}}^N \to \xi_{\lambda_k}$ strongly in $L^2([0,1])$, which yields
 
\begin{equation}\label{help3} \inner{\xi_{\lambda_{k}}^N}{x_{\lambda_{k}}^N} \to \inner{\xi_{\lambda_{k}}}{x_{\lambda_{k}}} \end{equation} 
for $N\to \infty$. Now consider \eqref{eq:cond_1} with $J = |.|_{TV}$, 
\begin{align*} 
D_J(x_{\lambda_{k+1}} -  x_{\lambda_{k}},x_{\lambda_{k+1}}) = 
J(x_{\lambda_{k+1}} -  x_{\lambda_{k}})  - J(x_{\lambda_{k}}) - \inner{\xi_{\lambda_{k}}}{-x_{\lambda_{k}}}. 
\end{align*}
Taking into account that $D_J(x_{\lambda_{k+1}}^N -  x_{\lambda_{k}}^N,x_{\lambda_{k+1}^N}) = 0$ by the results for discrete TV-denoising from above and using
the weak lower semicontinuity of $J$, one obtains 
\[ J(x_{\lambda_{k+1}} -  x_{\lambda_{k}})  \leq 
 \liminf\limits_{N\to \infty} J(x_{\lambda_{k+1}}^N -  x_{\lambda_{k}}^N) = 
  \liminf\limits_{N \to\infty} \left( J(x_{\lambda_{k}}^N) + \inner{\xi_{\lambda_{k}}^N}{-x_{\lambda_{k}}^N}\right) 
  =  J(x_{\lambda_{k}}) - \inner{\xi_{\lambda_{k}}}{x_{\lambda_{k}}},  
\]
where the last equality holds by \eqref{help1} and \eqref{help3}. 
Thus, the following (non-negative) Bregman distance satisfies
\[ D_J(x_{\lambda_{k+1}} -  x_{\lambda_{k}},x_{\lambda_{k+1}}) =  J(x_{\lambda_{k+1}} -  x_{\lambda_{k}})  -  J(x_{\lambda_{k}}) +  \inner{\xi_{\lambda_{k}}}{x_{\lambda_{k}}}  \leq 0, \] 
meaning that it must be $0$ and implying that \eqref{eq:cond_1} is
 verified.
\end{proof}

\subsection{TV-denoising in higher dimensions} 

The previous results  that the MHDM iteration 
agrees with Tikhonov regularization 
for denoising in the one-dimensional case cannot be 
extended to the higher-dimensional situation, not even in 
a discrete case. 
For TV-denoising on domains in $\R^2$ and $\R^3$ in a finite dimensional framework, the matrix
$D$ represents a discretization of the
gradient operator $D \sim \nabla$,   while
$D D^*$ is a discrete version of the operator 
$\nabla \div$, which does not necessarily satisfy the diagonal dominance condition of Lemma~\ref{lem0}. Thus, in general, the 
MHDM iteration does not agree with the Tikhonov regularization in higher dimensions. 
We can actually provide a counterexample to condition \eqref{eq:condition_2} in the two-dimensional case. 

\begin{example}\label{ex:520}
Let us consider the denoising problem in two dimensions
\eqref{eq:tikhonov} with $T $ being the identity.
Let $X$ be the space of $L^2(\R^2)$ functions $u$ 
with bounded (isotropic) TV-seminorm $J(u) < \infty$, where 
$J = |.|_{TV}$ on $\R^2$. We consider data $y$ given by 
the characteristic function of the unit square: $y = \chi_{[-1,1]^2}$. 
In this case, the explicit form of minimizers to \eqref{eq:tikhonov} 
are known, and they have as level sets “rounded” squares, i.e., a square whose 
edges are rounded by circular arcs; see, e.g., \cite{Alter05,Alter05II,Allard07}. 
Based on this, a useful explicit functional form has been stated in \cite[p.~1273]{Condat}, as described below. Namely,
 the minimizers of \eqref{eq:tikhonov} for $\lambda^{-1} < 1/(1 + \sqrt{\pi}/2)$ are defined as follows:
\begin{align} 
 x_\lambda(s,t)  :=
 \begin{cases} 1 -\lambda^{-1}(1 + \sqrt{\pi}/2) & r(s,t) \geq 1/(1 + \sqrt{\pi}/2) \\ 
  1 - \frac{\lambda^{-1}}{r(s,t)} &  \lambda^{-1} < r(s,t) < 1/(1 + \sqrt{\pi}/2)  \\
   0 & r \leq \lambda \text{ or }  (s,t) \not \in [0,1]^2,
 \end{cases}
\end{align}
where  $r(s,t):= 2 - |s|-|t| +\sqrt{2 (1-|s|)(1-|t|)}$. 
Of particular interest for us is the region 
\[ E_\lambda:= \{(s,t) \in (0,1)^2 \,|\, \lambda^{-1} < r(s,t) < 1/(1 + \sqrt{\pi}/2) \}, \]
which is bounded by two circular arcs and parts of the boundary of the unit square, where  the solution is smooth.   An illustration of $x_\lambda$ is given in Figure~\ref{fig:roundsquare}, where the graph of $x_\lambda$ at
$E_\lambda$ is marked in black. 
\begin{figure}
    \centering
    \includegraphics[width=8cm]{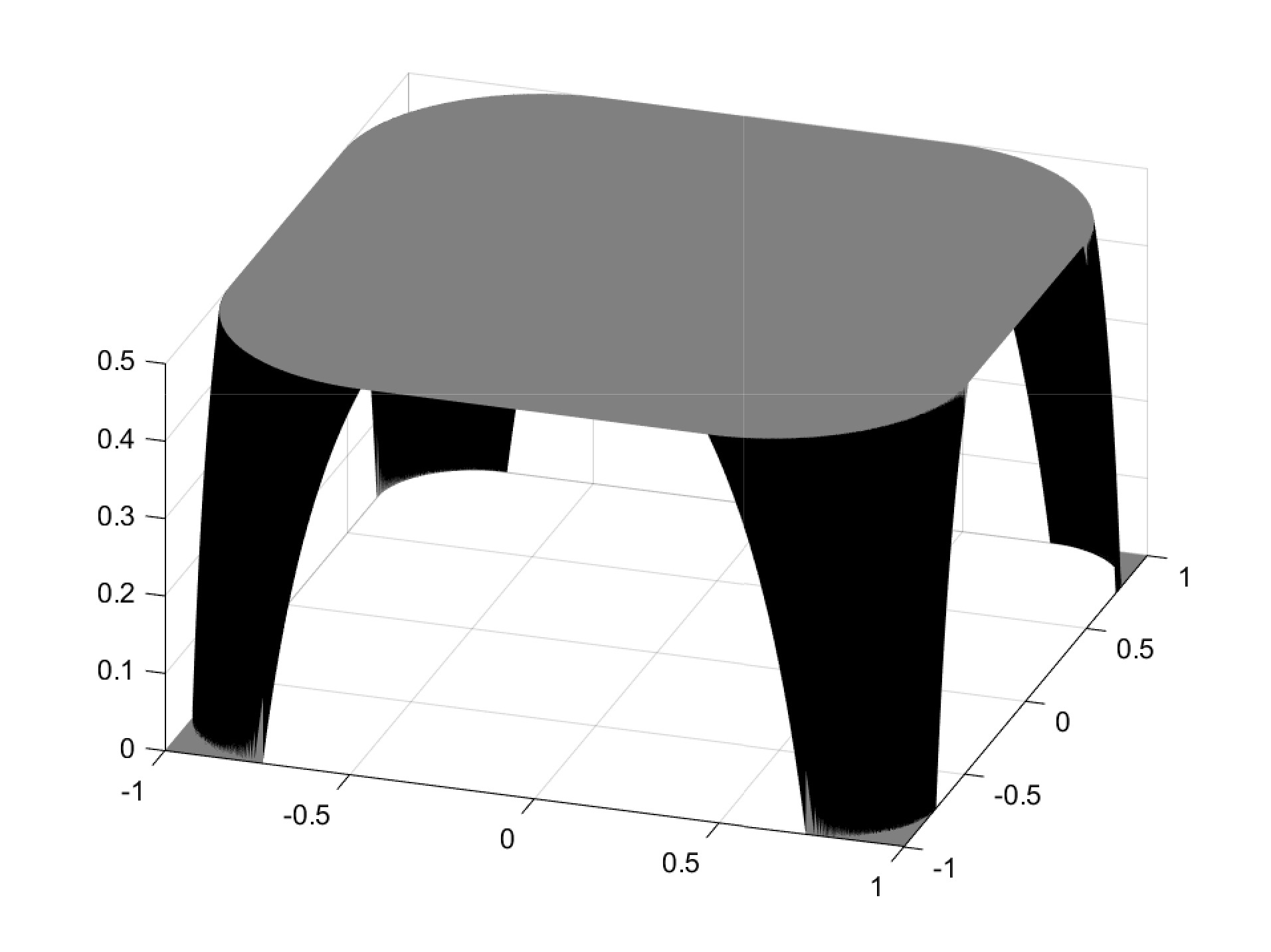}
    \caption{Illustration of 
    $x_\lambda$ in Example~\ref{ex:520}. The
    function at $ (s,t) \in E_\lambda$ is marked in black.
    }
    \label{fig:roundsquare}
\end{figure}

Now consider two solutions $x_{\lambda_{k+1}}$ and $x_{\lambda_k}$ with 
$\lambda_{k+1}> \lambda_{k} >  (1 + \sqrt{\pi}/2)$, and  take a fixed small ball $B_\eps$ with closure inside  the 
region $E_{\lambda_{k}}$, which is then also included in  $E_{\lambda_{k+1}}$.
It follows that 
\[ \nabla x_{\lambda_{k+1}} = -\lambda_{k+1}^{-1} \nabla \frac{1}{r(s,t)}
 \quad \mbox{and}\quad
 \nabla x_{\lambda_{k}} = -\lambda_{k}^{-1} \nabla \frac{1}{r(s,t)},  \qquad\mbox{for}\, (s,t) \in B_\eps.
\]
By smoothness, the $TV$-norm equals the $L^1$-norm of the gradient in $B_\eps$,  and it holds

\begin{align*} 
&\int_{B_\eps} |\nabla x_{\lambda_{k+1}} - \nabla x_{\lambda_{k}}| d(s,t) + 
\int_{B_\eps} |\nabla x_{\lambda_{k+1}}| d(s,t)  \\
& \qquad = 
|\lambda_{k+1}^{-1} -\lambda_k^{-1}| \int_{B_\eps}  \abs{\nabla \frac{1}{r(s,t)}} 
d(s,t)+ 
|\lambda_{k+1}^{-1}| \int_{B_\eps}  \abs{\nabla \frac{1}{r(s,t)}} d(s,t) \\
& \qquad = 
\left(|\lambda_{k+1}^{-1} -\lambda_k^{-1}| + |\lambda_{k+1}^{-1}|\right)\int_{B_\eps}  \abs{\nabla \frac{1}{r(s,t)}} d(s,t)  \\
&\int_{B_\eps} |2\nabla x_{\lambda_{k+1}} - \nabla x_{\lambda_{k}}| d(s,t)
= \left(|2\lambda_{k+1}^{-1} -\lambda_k^{-1}| \right)\int_{B_\eps}  \abs{\nabla \frac{1}{r(s,t)}} d(s,t).  
\end{align*} 
Since $\lambda_{k+1}^{-1} < \lambda_k^{-1}$, it follows 
that 
\[ \left(|\lambda_{k+1}^{-1} -\lambda_k^{-1}| + |\lambda_{k+1}^{-1}|\right)
= \lambda_k^{-1} 
>  |2\lambda_{k+1}^{-1} -\lambda_k^{-1}|,  \]
and thus the left-hand side in the above identity is strictly larger than the right one: 
\[ \int_{B_\eps} |\nabla x_{\lambda_{k+1}} - \nabla x_{\lambda_{k}}| d(s,t) + 
\int_{B_\eps} |\nabla x_{\lambda_{k+1}}| d(s,t)  > 
\int_{B_\eps} |2\nabla x_{\lambda_{k+1}} - \nabla x_{\lambda_{k}}| d(s,t). 
\]
As the solutions $x_{\lambda_{k+1}}$ and $x_{\lambda_k}$ are smooth in a neighborhood of $B_\eps$, we may 
decompose the $TV$-norm (cf.~\cite[Corollary~3.89]{Ambrosio00}) as 
\[ |x_{\lambda_{k+1}}|_{TV} =  \int_{B_\eps} |\nabla x_{\lambda_{k+1}}| d(s,t) + 
|x_{\lambda_{k+1}}|_{TV(\R^2\setminus B_\eps)}.
\]
One can proceed analogously for $x_{\lambda_{k}}$ and the combinations $x_{\lambda_{k+1}} - x_{\lambda_{k}}$ 
and $2 x_{\lambda_{k+1}} - x_{\lambda_{k}}$. Now considering 
\eqref{eq:condition_2}, it follows from the triangle inequality that 
\[| x_{\lambda_{k+1}} - x_{\lambda_{k}}|_{TV(\R^2\setminus B_\eps)}+ 
|x_{\lambda_{k+1}}|_{TV(\R^2\setminus B_\eps)}   \geq 
 |2\nabla x_{\lambda_{k+1}} - \nabla x_{\lambda_{k}}|_{TV(\R^2\setminus B_\eps)},  
\]
such that  we arrive at  
\[| x_{\lambda_{k+1}} - x_{\lambda_{k}}|_{TV}+ 
|x_{\lambda_{k+1}}|_{TV}   > 
 |2\nabla x_{\lambda_{k+1}} - \nabla x_{\lambda_{k}}|_{TV},  
\]
implying   that \eqref{eq:condition_2} does not hold. 
Consequently,  the MHDM iteration is not identical to Tikhonov regularization in this situation. 

Note however, that the 
set $E_\lambda$ that yields  a violation of \eqref{eq:condition_2} is rather narrow, such that the difference between the approximate solutions provided by  the two methods might be small. 
In fact, numerical experiments for this 
setup have only indicated a difference of less than 2\% (in the $L^2$-norm). 
\end{example}

On the other hand, for very special data $y$,
it is the case that two-dimensional TV-denoising agrees with the MHDM  iteration, namely when $y$ is the 
characteristic function of so-called calibrable sets, as explained below.

\begin{example}\label{ex:condition_TV_denoisng_geometric}
Let $\Omega \subset \R^2$ be bounded with Lipschitz boundary. We consider the case of total variation denoising, i.e. $X = L^2(\Omega)$, $J = \tv{\cdot}$ and $T = Id$. Let $C \subset \Omega$ be a convex set. We furthermore assume that $C$ has $C^{1,1}-$boundary and that the curvature $\kappa_{\partial C}$ of $C$ satisfies $\norm{\kappa_{\partial C}}_{L^\infty} \le q(C)$, where $q(C) = \frac{\tv{C}}{m(C)}$, with $m(C)$ being the Lebesgue measure of $C$. By Lemma 4 in \cite{TVflow}, this is equivalent to $C$ being convex and having the property that there is a an $\eps>0$, such that $C$ is the (possibly uncountable) union of balls with radius $\eps$. Define $f = b\, q(C)\chi_C$ for some $b \in \R$. By Theorem 4 and Proposition 7 in \cite{TVflow}, the minimizer  of the denoising problem is given by $x_\lambda = \sign(b)\max\set{b-\frac{1}{\lambda}),0}q(C)\chi_C$, that is 
\begin{equation*}
    x_{\lambda} = \begin{cases}
    (b-\frac{1}{\lambda})q(C)\chi_C & \text{ if }b>+\frac{1}{\lambda}\\
    (b+\frac{1}{\lambda})q(C)\chi_C & \text{ if }b<-\frac{1}{\lambda}\\
    0 & \text{ if } \abs{b} \le\frac{1}{\lambda}.
    \end{cases}
\end{equation*} 
Therefore, relation \eqref{eq:condition_2} holds by the same reasoning as in Example \ref{ex:verification_sparse_denoising}. 
\end{example}

\section[Examples and numerical results]{Examples and numerical results\footnote{The program code is available as ancillary file from the arXiv page of this paper.}}\label{sect:Numerical}
In this section, we provide examples of possible penalty functionals to be used for the MHDM.  We will focus on the comparison of the MHDM iteration to the generalized Tikhonov regularization. 

\subsection{One dimensional TV-regularization}
Let us first investigate the case of one-dimensional TV-regularization, which means considering the 
 functional \eqref{eq:Tikhonov_functional} on $L^2([0,1])$ with $J = \tv{\cdot}$.   We employ a discretization of the interval $[0,1]$ into $N = 100$ equidistant nodes and approximate the total variation via the discrete derivative operator, that is, $\tv{u}\approx \frac{1}{N}\norm{Du}_1$ with $D$ as in \eqref{eq:Discrete_derivative}. The ground truth is given as a piecewise constant signal 
$x^\dagger = \chi_{[0.3,0.5]} +\frac{1}{2}\chi_{[0.68,0.72]}$.

We deal with the cases of denoising, i.e., $T = Id$, and deblurring, where $T$ is a convolution operator of a centered Gaussian kernel with standard deviation $\sigma = 0.1$. The solutions for Tikhonov regularization and MHDM were computed using the primal dual algorithm of \cite{chambollepock}. A geometric progression of regularization parameters was used for MHDM: $\lambda_n = \lambda_0 10^n$ with $\lambda_0 =1$. To compare the two methods, the 
difference between the MHDM iterate  $x_n$ and the Tikhonov minimizer $x_{\lambda_n}$ at scale $\lambda_n$ was computed via $e_n = \norm{x_n-x_{\lambda_n}}_{L^2}$. The numerical results are presented in Table~\ref{tab:MHDMvsTIkhonov_1d-TV}.

\begin{table}[H]
    \centering
    \sisetup{round-mode=places,round-precision=4}
    \begin{tabular}{l|S|S}
    & \text{Denoising} & \text{Deblurring}  \\ 
    \text{k} & \text{ error $e_n$ Tik.-MHDM} & \text{ error $e_n$ Tik.-MHDM} \\ \hline \hline
        1 & 0.00102167013603803  & 0.00205171629785761  \\ \hline
        2 & 0.000515949307050504 & 0.0798600191290426  \\ \hline
        3 & 0.000490806048890083  & 0.115512594396305  \\ \hline
        5 & 0.000115238766956142 & 0.117920459233943 \\ \hline
        7 & 2.80324765348313e-06  & 0.111839192818755  \\ \hline
        11 & 2.82839879658074e-10  & 0.109341369607251  
    \end{tabular}
    \caption{\label{tab:MHDMvsTIkhonov_1d-TV} Comparison of MHDM iterates and Tikhonov minimizers for one-dimensional TV-regularization}
\end{table}
We note that the relative error in the denoising case is of magnitude of at most $10^{-3}$, while in the deblurring case it is of order $10^{-1}$ for $n\ge 3$. Up to numerical inaccuracies, this confirms the result from Theorem \ref{thm:1d_agreeing},  suggesting that Tikhonov regularization and the MHDM iteration disagree for the deconvolution problem. Furthermore, the numerical results indicate that  the MHDM iterates do not converge to the true solution $x^\dagger$ in the deblurring case. This is demonstrated in Figure \ref{fig:Deblurring} which displays $x_n$, as well as the one-step regularized solution with parameter $\lambda_n$ for $n = 4$ and $n = 11$.

\begin{figure}[H]
    \centering
    \includegraphics[scale = 0.5]{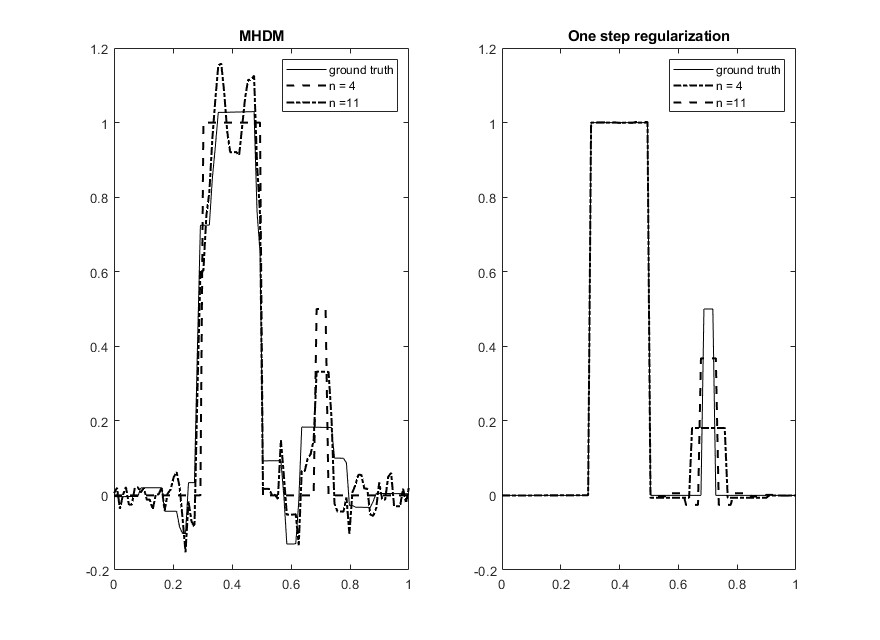}
    \caption{Comparison of different iterates of the MHDM and the corresponding one-step regularizers for TV-deblurring.
    }
    \label{fig:Deblurring}
\end{figure}

\subsection{\texorpdfstring{$\ell^p$}{TEXT}-regularization for \texorpdfstring{$p \in (0,1]$}{TEXT}}\label{ex_sparse}
In the regularization of sparsity constrained ill-posed problems, one often employs  Tikhonov regularization with $\ell^1$-penalty.  Another approach, even more sparsity promoting,  is to use the $\ell^p$-quasi-norms with $p \in (0,1)$ - see, e.g., \cite{grasmair2009, lorenz2008}.

In our experiments, we consider a sparse signal with peaks of different amplitudes. We apply a Gaussian convolution operator with standard distribution $0.025$  and add a normally distributed noise to create noisy data $f^\delta$ (cf.~Figure~\ref{fig:Truth_and_data}). We use the discrepancy principle for both the MHDM and the Tikhonov regularization. This means that we stop the iteration according to \eqref{eq:Discrepancy_principle} for the MHDM, while for solving the Tikhonov regularization problem we consider the same sequence of parameters and stop when the corresponding discrepancy principle condition  is satisfied. We choose $\tau  = 1.01$  in all experiments. 
\begin{figure}[H]
    \centering
  \includegraphics[scale = 0.55,trim = 1.5cm 0.9cm 1.5cm 0.5cm,clip]{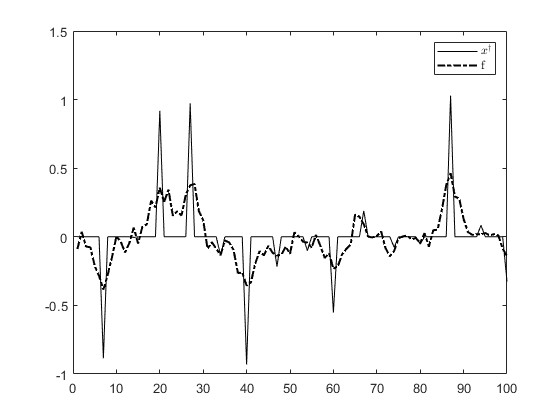}
    \caption{Ground truth $x^\dagger$ and observed data $f^\delta$.}
    \label{fig:Truth_and_data}
\end{figure}
In general, we do not expect the results of the MHDM to be significantly superior to those of the Tikhonov regularization. This can also be seen in Figures \ref{fig:MHDMvsTIkhonov_iterates} and \ref{fig:MHDMvsTIkhonov}.  Hence, we are more interested in how robust the algorithm is with regard to the involved parameters.
\begin{figure}[H]
    \centering
    \includegraphics[scale=0.3]{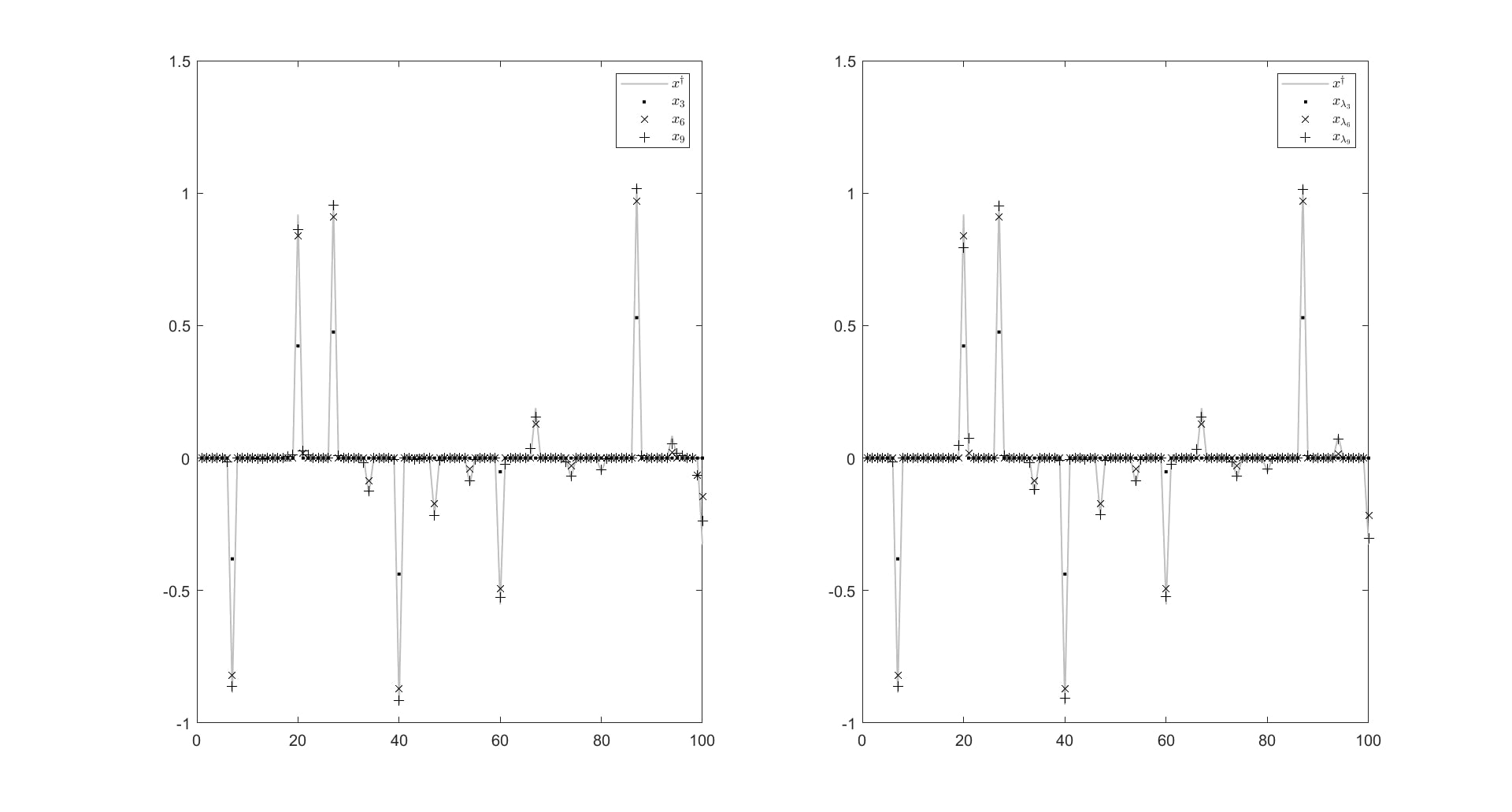}
    \caption{Iterates of the MHDM (left) and corresponding Tikhonov regularizers (right) with the parameters $\lambda_0 = 1$ and $\lambda_n = 2\lambda_{n-1}$.}
    \label{fig:MHDMvsTIkhonov_iterates}
\end{figure}

\begin{figure}[H]
    \centering
    \includegraphics[scale=0.3]{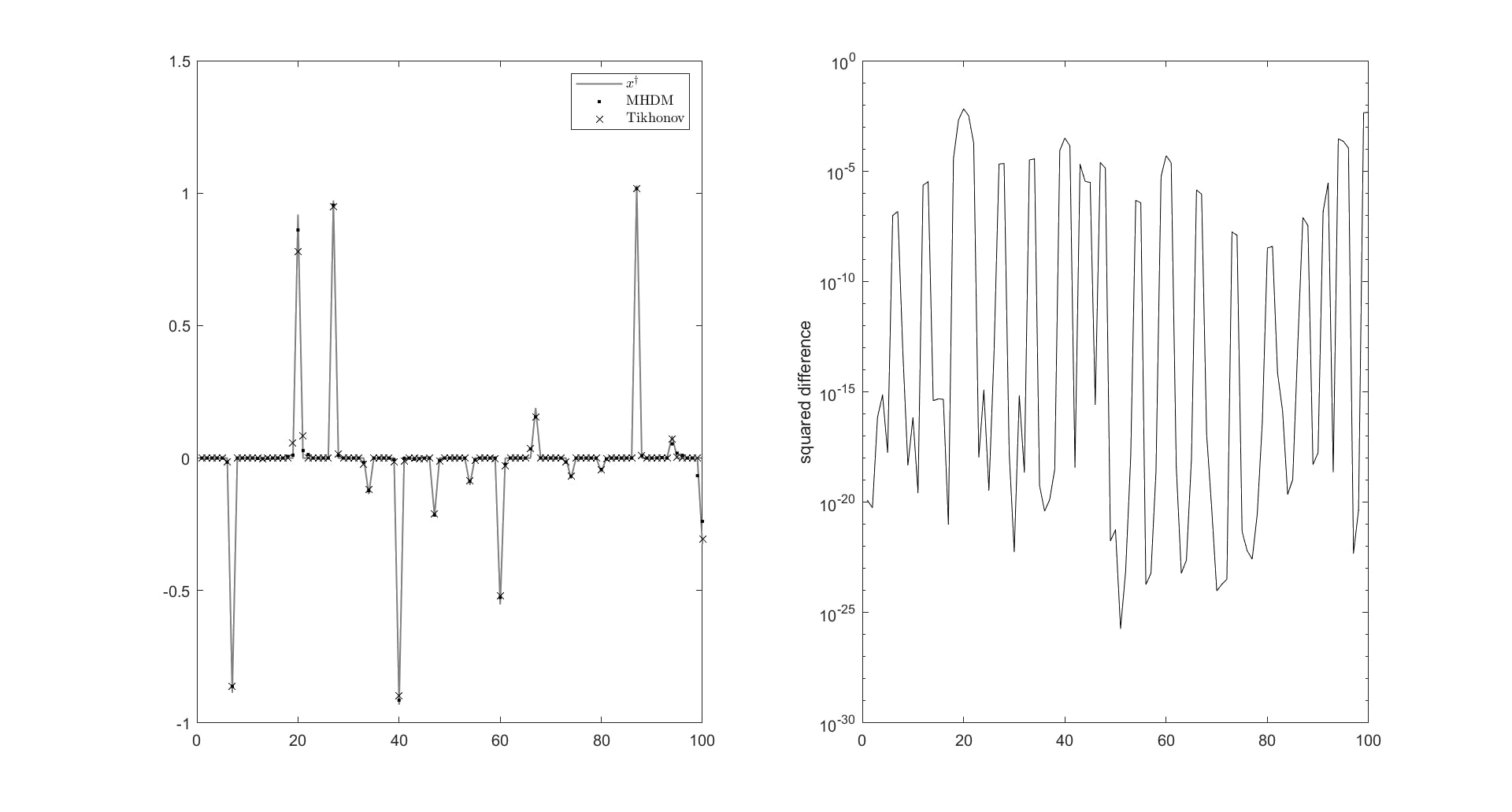}
    \caption{Comparison of the MHDM iterate and Tikhonov regularizer at the respective stopping index  (left) and componentwise squared error (right) with the parameters $\lambda_0 = 1$ and $\lambda_n = 2\lambda_{n-1}$.}
    \label{fig:MHDMvsTIkhonov}
\end{figure}

Let us start with the case of the $\ell^1$-penalty. All minimizers were computed using Nesterov's algorithm \cite{Nesterov1983AMF}. In Table \ref{tab:l1_different_noise} one can see the relative $\ell^2$-errors (that is $\norm{\tilde x -x^\dagger}_{\ell^2}\inv{\norm{x^\dagger}_{\ell^2}}$ if $\tilde x$ is an approximate solution obtained by either the MHDM or Tikhonov regularization) at the stopping index for different noise levels. For both  MHDM and Tikhonov regularization, we use a geometric progression $\lambda_n = 2^n \lambda_0 $ with $\lambda_0 = 1$.  Note that the Tikhonov regularization did not meet the discrepancy principle in the case of  the smallest noise level for any of the first $100$ tested parameters (see the $*$-entries in Table \ref{tab:l1_different_noise}).

\begin{table}[H]
    \centering
    \sisetup{table-column-width=12ex, round-mode=places,round-precision=4}
    \begin{tabular}{l||S|S|S}
         $\delta$ & 0.0051083329891876 & 0.050833298918755 & 0.508332989187551  \\ \hline \hline
        relative error MHDM & 0.058492868467940 & 0.069702351899454 & 0.445711040788777 \\ \hline
        relative error Tikhonov  & * & 0.088824779229680 & 0.604027245190495\\ \hline
        $n^*_\text{MHDM}$  & \sisetup{round-precision=0}\num{15} & \sisetup{round-precision=0}\num{9} & \sisetup{round-precision=0}\num{6}\\ \hline
        $n^*_\text{Tikhonov}$ & * & \sisetup{round-precision=0}\num{9} &\sisetup{round-precision=0}\num{6}  \\
                    \end{tabular}

    \caption{\label{tab:l1_different_noise} Relative errors under different levels of noise
    for the  $\ell^1$-penalty.}
\end{table}
While the number of minimizations used to meet the discrepancy principle is comparable, the MHDM performs slightly better than the Tikhonov regularization. Let us now investigate the stability of the algorithms with respect to the parameter choice. In Tables  \ref{tab:l1-different_lambda_0} and \ref{tab:l1_different_ratio} one can see the relative $\ell^2$-errors and stopping indices for varying initial values $\lambda_0$ and varying ratios for the geometric progression, respectively.

\begin{table}[H]
    \centering
    \sisetup{table-column-width=12ex, round-mode=places,round-precision=4}
    \begin{tabular}{l||S|S|S|S}
        $\lambda_0$ & \sisetup{round-precision=2}\num{0.01} & \sisetup{round-precision=1}\num{0.1} & \sisetup{round-precision=0}\num{1} & \sisetup{round-precision=0}\num{10} \\ \hline \hline 
        relative error MHDM & 0.052680985518303 & 0.072040027400777 & 0.069702351899454 & 0.053647413042756 \\ \hline
        relative error Tikhonov  & 0.098090439247152 & 0.080257469135223 & 0.088824779229680 & 0.097267029855473 \\ \hline
        $n^*_\text{MHDM}$   & \sisetup{round-precision=0}\num{16} & \sisetup{round-precision=0}\num{13} &\sisetup{round-precision=0}\num{9}  &\sisetup{round-precision=0}\num{6}  \\ \hline
      $n^*_\text{Tikhonov}$ & \sisetup{round-precision=0}\num{16} & \sisetup{round-precision=0}\num{12} & \sisetup{round-precision=0}\num{9} & \sisetup{round-precision=0}\num{6} \\

    \end{tabular}
    \caption{\label{tab:l1-different_lambda_0}Relative errors for different initial guesses $\lambda_0$ with $\lambda_n = 2^n\lambda_0$ and $\delta =0.0508$ for the  $\ell^1$-penalty.}
\end{table}

\begin{table}[H]
    \centering
    \sisetup{table-column-width=12ex, round-mode=places,round-precision=4}
    \begin{tabular}{l||S|S|S|S}

        ratio & \sisetup{round-precision=1}\num{1.2} & \sisetup{round-precision=0}\num{2} & \sisetup{round-precision=0}\num{3} & \sisetup{round-precision=0}\num{10} \\ \hline \hline
        relative error MHDM & 0.071740448268390 & 0.069702351899454 & 0.047284619801456 & 0.059375207394123 \\ \hline
        relative error Tikhonov  & 0.078881916529666 & 0.080257469135223 & 0.088824779229680 & 0.124386611653481 \\ \hline
        $n^*_\text{MHDM}$   & \sisetup{round-precision=0}\num{31} & \sisetup{round-precision=0}\num{9} & \sisetup{round-precision=0}\num{6} & \sisetup{round-precision=0}\num{4} \\ \hline
        $n^*_\text{Tikhonov}$ & \sisetup{round-precision=0}\num{30} & \sisetup{round-precision=0}\num{9} & \sisetup{round-precision=0}\num{6} & \sisetup{round-precision=0}\num{4} \\
        
    \end{tabular}
    \caption{\label{tab:l1_different_ratio}Relative errors for different geometric progressions with initial guess $\lambda_0 = 1$ and $\delta= 0.0508$ for the $\ell^1$-penalty.}
\end{table}

We observe that, in case $\lambda_n = 2^n \lambda_0$, the choice of the initial guess $\lambda_0$ is not   too important. However, varying the ratio $\frac{\lambda_{n+1}}{\lambda_n}$ leads to quite different results for Tikhonov regularization, while the MHDM behavior does not change significantly. Therefore, we argue that in the case of $\ell^1$-deblurring, the MHDM is a rather robust method, which in average seems to outperform Tikhonov regularization.

Let us now turn to $\ell^p$-regularization for $p \in (0,1)$, that is, consider the functional \begin{equation}\label{eq:functional_lp}
    J_p(x) = \sum_{i = 0}^ \infty\abs{x_i}^p.
\end{equation}
In order to compute minimizers of the generalized Tikhonov functional with this penalty term, we use the algorithm introduced in \cite{ghilli2019monotone}, whose Theorem 1 also ensures the well-definedness of the single step regularization and of the MHDM. Note that we may apply the same stopping rule \eqref{eq:Discrepancy_principle}, since the assumptions of part (i) in Theorem \ref{thm:Convergence_results} are satisfied with $C= 1$. Indeed, for any $p\in (0,1)$ and $x,y\ge 0$,  one has \begin{equation*}
    \abs{x-y}^p \le \abs{x}^p + \abs{y}^p.
\end{equation*} We first compare the  Tikhonov method with the MHDM for  $\lambda_n = \lambda_0 2^n$ with $\lambda_0 = 0.01$ and noise level $\delta= 0.0508$, while  allowing different values for  $p$. Furthermore, we  consider a version of the flexible  MHDM \eqref{step flexible} employing  $J_n$ as in \eqref{eq:functional_lp}  with a variable $p_n$ instead of a fixed $p$ in each iteration, namely for an increasing sequence $p_n = 0.95-\frac{0.9}{n+1}$ and then for a decreasing sequence $p_n = 0.05 + \frac{0.9}{n+1}$.  The results of both experiments can be found in Tables \ref{tab:different_p} and \ref{tab:varying_p}

\begin{table}[H]
    \begin{center}
    \sisetup{table-column-width=9.5ex, round-mode=places,round-precision=4}
    \begin{tabular}{l||S|S|S|S|S|S}
        $p$ & \sisetup{round-precision=3}\num{0.995} & \sisetup{round-precision=1}\num{0.9} & \sisetup{round-precision=2}\num{0.75} & \sisetup{round-precision=1}\num{0.5} & \sisetup{round-precision=2}\num{0.25} & \sisetup{round-precision=2}\num{0.05} \\ \hline\hline
        relative error MHDM & 0.051164595389368 & 0.049193565232175 & 0.039524098404196 & 0.059414025407425 & 0.033106651955753 & 0.049441591949795 \\ \hline
        relative error Tikhonov & 0.089019478446407& 0.021051640596944 & 0.013158632007887 & 0.010846496341248 & 0.010856320616812 & 0.010802704710041 \\ \hline
        $n^*_\text{MHDM}$ &\sisetup{round-precision=0}\num{16}  &\sisetup{round-precision=0}\num{16}  &\sisetup{round-precision=0}\num{16}  &\sisetup{round-precision=0}\num{19}  &\sisetup{round-precision=0}\num{20}  & \sisetup{round-precision=0}\num{16} \\ \hline
        $n^*_\text{Tikhonov}$ &\sisetup{round-precision=0}\num{16}   & \sisetup{round-precision=0}\num{16}  &\sisetup{round-precision=0}\num{17}   &\sisetup{round-precision=0}\num{17}   &\sisetup{round-precision=0}\num{17}   &\sisetup{round-precision=0}\num{18}   \\

    \end{tabular}
    \caption{Relative errors for different choices of $p$.}
    \label{tab:different_p}
    \end{center} 
\end{table}

\begin{table}[H]
    \centering
    \begin{tabular}{l||l|l}
        
                 &  $p_n$ increasing &$p_n$ decreasing\\ \hline \hline
         relative error& 0.0489 & 0.0533\\\hline
         $n^*$ & 15& 18\\ \hline
         $p_{n^*}$ & 0.1500 &0.8643
    \end{tabular}
    \caption{Relative errors for MHDM with varying penalty terms $J_{p_n}$.}
    \label{tab:varying_p}
\end{table}
Once again, the number of minimizations until the discrepancy principle is satisfied is very similar for both MHDM and  Tikhonov regularization. For $p$ close to $1$, the MHDM seems to produce slightly better results, while for smaller values of $p$ Tikhonov regularization seems to be superior. Those results achieved by Tikhonov regularization are also the overall best ones. The more general approach with functionals $J_{p_n}$ did not show very different results from the approach with fixed exponent. Nevertheless, a version with adaptive penalty terms would be an interesting concept for further research.
For variations of $\lambda_0$ and of the ratio $\frac{\lambda_n}{\lambda_{n+1}}$ defining the parameters $\lambda_n$ in the case of fixed exponent $p$, we observe that the MHDM performs again very similarly. The Tikhonov regularization performs more stable than in the $\ell^1$ case, though  it is outperformed by the MHDM for large ratios $\pars{\frac{\lambda_{n+1}}{\lambda_n}\approx 60}$. We expect that for ill-posed problems with higher degree of ill-posedness than the one we considered, the outperformance will occur for smaller ratios. Thus, we conclude that by applying the discrepancy principle, both methods seem to produce comparable reconstructions of the true data, but the MHDM is less sensitive to parameter choices.

\section{Conclusion}
We analyze the Multiscale Hierarchical Decomposition Method (MHDM) involving various convex and nonconvex penalties in a general function space framework and provide sufficient conditions for the convergence of the residual. We also provide a counterexample for which the residual does not converge, while the sufficient conditions are not satisfied either. Then, we  extend the MHDM to adaptive regularization functionals, showing an interesting multiscale norm decomposition of the data. This applies in particular to the  Bregman iteration method, thus leading to a new result in this respect. Furthermore, we propose a characterization for the generalized Tikhonov regularization at a given scale to agree with the MHDM. We provide a sufficient condition for the agreement in finite dimensional $\ell^1$-regularization and use it to prove that the MHDM and Tikhonov regularization are identical for $1$-dimensional TV-denoising. Moreover, we test the MHDM for sparsity constrained deconvolution problems and find it to be stable with regard to the involved parameters. Conditions for the convergence of the MHDM iterates, as well as convergence rates remain open questions. 

\section{Acknowledgments}
The authors are grateful to Daria Ghilli (University of Pavia) for providing an initial version of a code for nonconvex sparsity regularization. E. Resmerita and T. Wolf are supported by the Austrian Science Fund (FWF): DOC 78. The constructive remarks of the referees are appreciated, as they led to improving the presentation of the manuscript.

\printbibliography

\end{document}